\documentclass[12pt, a4paper]{amsart}
\usepackage{amssymb,amsmath,amsthm}

\usepackage{comment}
\excludecomment{draft}

\newtheorem{theorem}{Theorem}
\newtheorem{lemma}[theorem]{Lemma}
\newtheorem{corollary}[theorem]{Corollary}

\newtheorem{proposition}[theorem]{Proposition}

\theoremstyle{remark}

\newtheorem*{remark}{Remark}
\newtheorem*{example}{Example}
\numberwithin{theorem}{section} \numberwithin{equation}{section}

\newcommand{\Mp}{\text {\rm Mp}}

\newcommand{\R}{\mathbb{R}}
\newcommand{\C}{\mathbb{C}}

\newcommand{\Q}{\mathbb{Q}}

\newcommand{\Z}{\mathbb{Z}}
\newcommand{\N}{\mathbb{N}}
\newcommand{\SL}{{\text {\rm SL}}}

\newcommand{\sgn}{\operatorname{sgn}}
\newcommand{\PSL}{{\text {\rm PSL}}}

\newcommand{\F}{\mathcal{F}}

\newcommand{\h}{\mathbb{H}}
\newcommand{\G}{\Gamma}
\newcommand{\dg}{\mathcal{D}} 
\newcommand{\dgdelta}{{\mathcal{D}(\Delta)}} 
\newcommand{\abs}[1]{\left\vert#1\right\vert}
\newcommand{\e}{\mathfrak{e}}

\newcommand{\smallabcd}{\left(\begin{smallmatrix}a & b \\ c & d\end{smallmatrix}\right)}
\newcommand{\Deltaover}[1]{\left(\frac{\Delta}{#1}\right)}

\newcommand{\smallTmatrix}{\left(\begin{smallmatrix}1 & 1 \\ 0 & 1\end{smallmatrix}\right)}
\newcommand{\smallSmatrix}{\left(\begin{smallmatrix}0 & -1 \\ 1 & 0\end{smallmatrix}\right)}

\newcommand{\M}{\mathcal{M}}

\newcommand{\mt}{\mathbf{t}}
\newcommand{\thetaL}[1]{\Theta_{\Delta,r}(\tau,z,#1)}
\newcommand{\phikm}{\varphi_{\text{KM}}}
\newcommand{\phis}{\varphi_{\text{S}}}

\newcommand{\thetah}[1]{\theta_h(\tau,z,#1)}

\newcommand{\LO}{\Lambda^{\text{o}}_{\Delta,r}(\tau,f)}
\newcommand{\LE}{\Lambda^{\text{e}}_{\Delta,r}(\tau,f)}

\begin{document}
\title[] {Formulas for the coefficients of half-integral weight harmonic Maa\ss{}  forms}

\author{Claudia Alfes}
\address{Fachbereich Mathematik, Technische Universit\"at Darmstadt, Schlo\ss gartenstr. 7, D-64289 Darmstadt, Germany} \email{alfes@mathematik.tu-darmstadt.de}
\thanks {} \subjclass[2000] {}
\date{\today}

\begin{abstract}
Recently, Bruinier and Ono proved that the coefficients of certain weight $-1/2$ harmonic weak Maa\ss{}  forms are given as ``traces'' of singular moduli for harmonic weak Maa\ss{} forms. Here, we prove that similar results hold for the coefficients of harmonic weak Maa\ss{}  forms of weight $3/2+k$, $k$ even, and weight $1/2-k$, $k$ odd, by extending the theta lift of Bruinier-Funke and Bruinier-Ono. Moreover, we generalize these results to include \textit{twisted} traces of singular moduli using earlier work of the author and Ehlen on the twisted Bruinier-Funke-lift. Employing a duality result between weight $k$ and $2-k$, we are able to cover all half-integral weights. 
We also show that the non-holomorphic part of the theta lift in weight $1/2-k$, $k$ odd, is connected to the vanishing of the special value of the $L$-function of a certain derivative of the lifted function.
\end{abstract}

\maketitle

\section{Introduction and statement of results}

A classical result states that the values of the modular $j 
$-invariant at quadratic irrationalities, called ``singular moduli'', are  
algebraic integers. Zagier \cite{ZaTr} showed that the (twisted) traces  
of these values occur as the  Fourier coefficients of weakly holomorphic modular forms of weight  
$3/2$. These are meromorphic modular forms that are holomorphic on the  
upper half-plane $\mathbb{H}:=\left\{z\in\C;\Im(z)>0\right\}$, with possible  
poles at the cusps.

Zagier's results were generalized in various directions, mostly for  
modular curves of genus 0 \cite{BriOnoPoin,DukeJenkins,Kim,MilPix}. Building upon previous work of Funke  
\cite{Funke} Bruinier and Funke \cite{BrFu06} showed that Zagier's result on the traces of the $j$-function can be obtained as a special case of a theta lift  using a kernel  
function constructed by Kudla and Millson \cite{KM86}. The lift of Bruinier and Funke maps a harmonic weak Maa\ss{}  form $F$ of weight 0 on a modular curve of arbitrary genus to a form $\Lambda(\tau,F)$ of weight $3/2$. The Fourier coefficients of positive index of $\Lambda(\tau,F)$ are given by the traces of the CM-values of $F$ \cite[Theorem 7.8]{BrFu06}. In \cite{AE} a twisted version of Bruinier's and Funke's theta lift was  
considered.

Recently, Bruinier and Ono \cite{BrOno2} obtained a result similar to that of Bruinier and Funke for  
the coefficients of weight $-1/2$ harmonic Maa\ss{} forms. Using the Maa\ss{}  raising and lowering operators they modified the theta lift of Bruinier and  
Funke such that it lifts from weight $-2$ to weight $-1/2$. In this  
way, they obtained a closed formula for the partition function $p(n)$  
in terms of traces of the CM-values of the derivative of a weakly holomorphic modular form $F$ of  
weight $-2$ on $\G_0(6)$. Moreover, they showed that these values are algebraic with explicitly bounded denominators.

Here, we generalize their construction in two ways. Firstly, we extend the lift to other weights (as suggested by Bruinier and Ono) and secondly, we include twisted traces.

To make these results more precise we briefly introduce harmonic weak Maa\ss{}  forms and quadratic forms. Throughout the introduction $N$ is a square-free positive integer.

Let $k$ be an integer. A twice continuously differentiable function $f:\h \rightarrow \C$ 
is called a \textit{harmonic weak Maa\ss{}  form of weight $k$ with respect to $\G_0(N)$} 
if it satisfies
\begin{enumerate}
 \item $f(\gamma \tau) = (c\tau+d)^k f(\tau)$, for $\gamma \in \G_0(N)$,
\item $\Delta_k f=0$,
 \item there is a polynomial $P_f(\tau)=\sum_{n\leq 0}c^+(n)q^n\in\C[q^{-1}]$ such that $f(\tau)-P_f(\tau)\rightarrow 0$ as $v\rightarrow \infty$. Similar conditions are required at all cusps.
 \end{enumerate}
Here, we write $\tau=u+iv$ with $u,v \in \R$, and
$\Delta_k=-v^2\left(\frac{\partial^2}{\partial u^2}+\frac{\partial^2}{\partial v^2}\right)
+ikv\left(\frac{\partial}{\partial u}+i\frac{\partial}{\partial v}\right)$ 
is the weight $k$ Laplace operator.
  
The space of harmonic Maa\ss{}  forms includes the spaces of (weakly) holomorphic modular forms and cusp forms. We will also define vector valued analogs of these spaces. In this case, we require a transformation property with respect to the metaplectic group (see Section \ref{sec:weil}). Moreover, we define forms of half-integral weight, then requiring that $4|N$. 
Throughout the introduction we only consider harmonic weak Maa\ss{}  forms satisfying a strong growth condition at the cusps, whereas we will later allow weaker growth conditions. 

The Fourier expansion at $\infty$ of any harmonic weak Maa\ss{}  form uniquely decomposes into a holomorphic and a non-holomorphic part \cite[Section 3]{BrFu04}
\[
f(\tau)=\sum_{n\gg -\infty}c^{+}(n)q^n +  \sum\limits_{n <0}c^{-}(n)\G(1-k,4\pi \abs{n}v)q^n,
\]
where $\G(a,x)$ denotes the incomplete $\G$-function and $q:=e^{2\pi i \tau}$. 

For a negative discriminant $D$ and an integer $r$ with $r^2\equiv D \pmod{4N}$, we denote by $\mathcal{Q}_{D,r,N}$ the set of positive and negative definite integral binary quadratic forms $Q=[a,b,c]$ of discriminant $D=b^2-4ac$ with $N|c$ and $b\equiv r\pmod{2N}$. Then we let $\alpha_Q=\frac{-b+\sqrt{D}}{2a}$ be the Heegner point corresponding to $Q=[a,b,c]\in\mathcal{Q}_{D,r,N}$. The group $\G_0(N)$ acts on $\mathcal{Q}_{D,r,N}$ with finitely many orbits.

We let $\Delta $ be a fundamental discriminant and  $r\in\Z$ with $r^2\equiv \Delta \pmod{4N}$. Moreover, let $d$ be a positive integer such that $-\sgn(\Delta)d$ is a square modulo $4N$ and let $h\in \Z/ 2N\Z$ (this index will determine the component of a vector valued form). For a harmonic Maa\ss{}  form $f$ of weight $-2k$ for $\G_0(N)$ we put $\partial f := R_{-2k}^k f$, where $R_{-2k}^k= R_{-2} \circ \cdots\circ R_{-2k}$ and $R_k:= 2i \frac{\partial}{\partial \tau}+k v^{-1}$ is the Maa\ss{}  raising operator. We define the modular trace function
\begin{equation*}
 \mt_{\Delta,r}(f;d,h)=\sum\limits_{Q\in\G_0(N)\backslash\mathcal{Q}_{-d\abs{\Delta},rh,N}}\frac{\chi_{\Delta}(Q)}{\abs{\overline\G_Q}}\partial f(\alpha_Q).
\end{equation*}
Here $\overline{\G}_Q$ denotes the stabilizer of $Q$ in $\overline{\G_0(N)}$, the image of $\G_0(N)$ in $\PSL_2(\Z)$.
The function $\chi_\Delta$ is a genus character, which will be defined in Section \ref{sec:twistvec}.
For $\Delta=1$ we have $\chi_\Delta(Q)=1$ for all $Q \in \mathcal{Q}_{-d\abs{\Delta} ,rh,N}$.

Defining a theta lift using a twisted Kudla-Millson theta function we can then generalize the results of Bruinier and Funke \cite{BrFu06} and Bruinier and Ono \cite{BrOno2}.

\begin{theorem}\label{intro:thm1}
Assume the notation as above and let $f$ be a harmonic weak Maa\ss{} form of weight $-2k$ for $\G_0(N)$.
\begin{enumerate}
 \item 
For even $k$ there is a vector valued weakly holomorphic modular form of weight $3/2+k$ such that its $(d,h)$-th Fourier coefficient is given by
\[
 (-4\pi d)^{k/2} \mt_{\Delta,r}(f;d,h).
\]

 \item
For odd $k$ there is a vector valued harmonic weak Maa\ss{}  form of weight $1/2-k$ such that its $(d,h)$-th Fourier coefficient is given by
\[
 \left(\frac{1}{4\pi d}\right)^{(k+1)/2} \prod_{j=0}^{(k-1)/2}\left(\frac{k}{2}+j\right)\left(j-\frac{k+1}{2}\right) \mt_{\Delta,r}(f;d,h).
\]
This form is weakly holomorphic if and only if $f$ is weakly holomorphic or if the twisted $L$-function of $\xi_{-2k} (f):=2iv^{-2k}\overline{\frac{\partial}{\partial\bar{\tau}}f}=\sum_{n=1}^\infty a_{\xi_{-2k}(f)}(n)q^n $ vanishes at $k+1$, that is  
\[
L(\xi_{-2k}(f),\Delta, k+1)=\sum_{n=1}^\infty \left(\frac{\Delta}{n}\right)a_{\xi_{-2k}(f)}(n)n^{-(k+1)}= 0.
\]
\end{enumerate} 
\end{theorem}

\begin{remark}
 Note that the transformation properties of the vector valued form depend on the sign of $\Delta$ (see Corollary \ref{cor:lifttrans}).
\end{remark}
\begin{remark}
The theorem sheds new light on nonvanishing conditions for the twisted central $L$-values. Goldfeld \cite{Goldfeld} conjectured that a positive proportion of discriminants $0<\abs{\Delta}\leq X$ have the property that $L(\xi_{-2k}(f),\Delta, k+1)\neq 0$ (where the notation is as in Theorem \ref{intro:thm1}). Note that $\xi_k$ is surjective \cite[Theorem 3.7]{BrFu04}. It was shown by Ono and Skinner \cite{OnoSkin} that
\[
 \#\left\{\abs{\Delta}\leq X;\, L(\xi_{-2k}(f),\Delta, k+1) \neq 0 \text{ and } \gcd(\Delta,N)=1\right\} \gg \frac{X}{\log X}.
\]
This result was later improved by Ono in \cite{OnoNonvan}.
\end{remark}

Similarly to the formula Bruinier and Ono get for the partition function $p(n)$ \cite[Theorem 1.1]{BrOno2}, whose generating function is essentially the reciprocal of Dedekind's $\eta$-function, we obtain a formula in terms of traces of a weight $-26$ function $F$ for the coefficients of $\frac{1}{\eta(\tau)^{25}}$ (see Section \ref{sec:eta}). Note, however, that we do not obtain the algebraicity of the CM-values of $F$ since this function is not weakly holomorphic (compare Theorem \ref{intro:thm2}).

\begin{example}
We consider
$ \frac{1}{\eta(\tau)^{25}}=q^{-25/24}\prod_{n=1}^\infty (1-q^n)^{-25}$.
For $n>0$ the coefficient of index $\frac{24n-1}{24}$ of this function is given by
\begin{align*}
& -\frac{185725}{4429185024\pi^{13}}\left(\frac{1}{24n-1}\right)^7
\\
&\quad\quad\quad\quad\quad\quad\quad\times\left(\mt_{1,1}(F;24n-1,1)+\mt_{1,1}(\tilde{F};24n-1,1)\right),
\end{align*}
where $F$ and $\tilde{F}$ are harmonic weak Maa\ss{}  forms of weight $-26$ given by a linear combination of Poincar\'{e} series defined in Section \ref{sec:eta}. Note that even though $F$ and $\tilde{F}$ are not weakly holomorphic, the sum of these traces is an integer (up to a constant).  
\end{example}

\begin{remark}
 Given a scalar valued form one has to realize it as the component of a vector valued harmonic weak Maa\ss{}  form to be able to obtain a formula as above for its coefficients. Here, the vector valued form has to transform with respect to a certain Weil representation as in Section \ref{sec:weil}.
 For a detailed discussion of this problem see a recent preprint of Fredrik Str\"omberg \cite{Freddy}. If the corresponding vector valued form is known, one can construct the input function using Poincar\'{e} series. 
\end{remark}

Moreover, using work of Bruinier and Ono \cite[Section 4]{BrOno2}, we bound the denominators of the singular moduli. 
\begin{theorem}\label{intro:thm2}
 Assume that the weakly holomorphic modular form $f$ of weight $-2k$ for $\G_0(N)$ has integral coefficients at all cusps. Let $D>0$ be coprime to $2N$ and $r\in\Z$ with $r^2\equiv D \pmod{4N}$. If $Q\in \mathcal{Q}_{D,r,N}$ is primitive and positive definite, then $3^{k} D^k (\frac{1}{\pi})^k R_{-2k}^k f(\alpha_Q)$ is an algebraic integer in the ring class field for the order $\mathcal{O}_D \subset \Q(-\sqrt{D})$.
The multiset of values  $ R_{-2k}^k f(\alpha_Q)$ is a union of Galois orbits.
\end{theorem}

\begin{remark}
 Note that these results generalize works of Duke and Jenkins \cite{DukeJenkins} and Miller and Pixton \cite{MilPix}. We also obtain the results in Section 9 of \cite{ZaTr} as a special case.
\end{remark}

We let $\kappa= 3/2+k$, if $k$ is odd, and $\kappa=1/2-k$, if $k$ is even.
Using a duality result we also realize the Fourier coefficients of harmonic weak Maa\ss{}  forms of weight $\kappa$ as traces of CM-values of weight $-2k$ harmonic weak Maa\ss{}  forms.

\begin{theorem}\label{intro:thm3}
We let $F$ be a harmonic weak Maa\ss{}  form of weight $\kappa$ and $\Delta,d,r$ as before. We denote the $(d,h)$-th Fourier coefficient of the holomorphic part by $c^+_F(d,h)$. Moreover, let $f\in M^{\text{!}}_{-2k}(N)$, such that the lift of $f$ is weakly holomorphic. We denote the coefficients of the principal part of the theta lift by $a_{\Lambda}^+(d,h)$. Then we have
 \begin{align*}
 \sum_{h\in \Z/2N\Z}& \sum_{\substack{d\geq 0\\d\equiv \sgn(\Delta)h^2/4N\, (\Z)}}  c_F^+(-d,h) \mt_{\Delta,r}(f;d,h)
\\
&=-\sum_{h\in \Z/2N\Z}\sum_{\substack{d\geq 0\\ -N\abs{\Delta}d^2\equiv \sgn(\Delta)h^2/4N\, (\Z)}}c_F^+(N\abs{\Delta}d^2,h) a_{\Lambda}^+(-N\abs{\Delta}d^2,h).
\end{align*}
\end{theorem}

 The theorem follows from the fact that the theta lift is orthogonal to cusp forms (see Section \ref{sec:orthcusp}). The well-known pairing between weight $k$ and $2-k$ forms introduced by Bruinier and Funke \cite[Proposition 3.5]{BrFu04} directly implies the formula in Theorem \ref{intro:thm3}.

\begin{example}
We let $f(q)$ be Ramanujan's mock theta function 
\[
f(q):= 1+ \sum\limits_{n=1}^\infty \frac{q^{n^2}}{\left(1+q\right)^2 \left(1+q^2\right)^2\cdots\left(1+q^n\right)^2 }= 1+ \sum\limits_{n=1}^\infty a_f(n)q^n.
\] 

Let $F(z)$ be the weakly holomorphic modular function for $\G_0(6)$ defined by
\[
F(z)= -\frac{1}{40} \frac{E_4(z)+ 4E_4(2z)-9E_4(3z)- 36E_4(6z)}{\eta(z)^2\eta(2z)^2\eta(3z)^2\eta(6z)^2}= q^{-1}-4+83q+\cdots,
\]
where $E_4(z)$ is the usual weight $4$ Eisenstein series. 
For a fundamental discriminant $\Delta<0$ with $\Delta\equiv 1 \pmod{24}$ we then have
\begin{align*}
 a_f\left(\frac{\abs{\Delta}+1}{24}\right) = -\frac{1}{8i\sqrt{\abs{\Delta}}}&(\mt_{\Delta,1}(F;1,1) -\mt_{\Delta,1}(F;1,5)
\\
&\quad\quad+\mt_{\Delta,1}(F;1,7)-\mt_{\Delta,1}(F;1,11)).
\end{align*}
Since $F$ is weakly holomorphic, it follows by classical results that the values $F(\alpha_Q)$ as $Q$ ranges over $\mathcal{Q}_{\Delta,1,6}$ are algebraic numbers.
\end{example}

The paper is organized as follows. In Section \ref{sec:prelim} we review background material on quadratic spaces and (vector valued) automorphic forms. Moreover, we review the strategy for twisting vector valued automorphic forms as developed in \cite{AE}. Then we define differential operators on spaces of automorphic forms and Poincar\'{e} series. The twisted Kudla-Millson and Siegel theta function are defined in Section \ref{sec:theta}. In Section \ref{sec:thetalift} we then define the theta lift and compute the lift on Poincar\'{e} series and the coefficients of positive index of the holomorphic part of the lift. Then we show that the lift is orthogonal to cusp forms which essentially implies Theorem \ref{intro:thm3}. Theorem \ref{intro:thm2} is proven in Section \ref{sec:singmod}. The two examples of the introduction are then derived in detail in the last section.

\section*{Acknowledgements}
I would like to thank Jan Bruinier for his constant support, his advice, and many fruitful conversations. Moreover, I am grateful to Ken Ono for his advice and comments on the paper. I am also indebted to Stephan Ehlen for his comments on this paper and many discussion. Especially, the generalization to twisted traces would not exist without him.

\section{Preliminaries}\label{sec:prelim}

For a positive integer $N$ we consider the rational quadratic space of signature $(1,2)$ given by
\[
V:=\left\{\lambda=\begin{pmatrix} \lambda_1 &\lambda_2\\ \lambda_3& -\lambda_1\end{pmatrix}; \lambda_1,\lambda_2,\lambda_3 \in \Q\right\}
\]
and the quadratic form $Q(\lambda):=N\text{det}(\lambda)$.
The associated bilinear form is $(\lambda,\mu)=-N\text{tr}(\lambda\mu)$ for $\lambda,\mu \in V$.

We let $G=\mathrm{Spin}(V) \simeq \SL_2$, viewed as an algebraic group over $\Q$
and write $\overline\G$ for its image in $\mathrm{SO}(V)\simeq\mathrm{PSL}_2$.
By $D$ we denote the associated symmetric space. It can be realized as the Grassmannian of lines 
in $V(\R)$ on which the quadratic form $Q$ is positive definite,
\[
D \simeq \left\{z\subset V(\R);\ \text{dim}z=1 \text{ and } Q\vert_{z} >0 \right\}.
\]
Then the group $\SL_2(\Q)$ acts on $V$ by conjugation
\[
  g\textbf{.}V :=g \lambda g^{-1},
\]
for $\lambda \in V$ and $g\in\SL_2(\Q)$. In particular, $G(\Q)\simeq\SL_2(\Q)$.

We identify the symmetric space $D$ with the complex upper half-plane $\h$ in the usual way,
and obtain an isomorphism
between $\h$ and $D$ by
\[
 z \mapsto \R \lambda(z),
 \]
where, for $z=x+iy$, we pick as a generator for the associated positive line
 \[
 \lambda(z):=\frac{1}{\sqrt{N}y} \begin{pmatrix} -(z+\bar{z})/2 &z\bar{z} \\  -1 & (z+\bar{z})/2 \end{pmatrix}.
 \]
The group $G$ acts on $\h$ by linear fractional transformations and
the isomorphism above is $G$-equivariant.
Note that $Q\left(\lambda(z)\right)=1$ and $g\textbf{.}\lambda(z)=\lambda(gz)$ for $g\in G(\R)$.
Let $(\lambda,\lambda)_z=(\lambda,\lambda(z))^2-(\lambda,\lambda)$. This is the minimal majorant of $(\cdot,\cdot)$ associated with $z\in D$.

We can view $\G_0(N)$ as a discrete subgroup of $\mathrm{Spin}(V)$ and we write $M=\G_0(N) \setminus D$ for the attached locally symmetric space.

We identify the set of isotropic lines $\mathrm{Iso}(V)$ in $V(\Q)$
with $P^1(\Q)=\Q \cup \left\{ \infty\right\}$ via
\[
\psi: P^1(\Q) \rightarrow \mathrm{Iso}(V), \quad \psi((\alpha:\beta))
 = \mathrm{span}\left(\begin{pmatrix} \alpha\beta &\alpha^2 \\  -\beta^2 & -\alpha\beta \end{pmatrix}\right).
\]
The map $\psi$ is a bijection and $\psi(g(\alpha:\beta))=g.\psi((\alpha:\beta))$. 
Thus, the cusps of $M$ (i.e. the $\G_0(N)$-classes of $P^1(\Q)$) can be identified with the $\G_0(N)$-classes of $\mathrm{Iso}(V)$.

If we set $\ell_\infty := \psi(\infty)$, then $\ell_\infty$ is spanned by 
$\lambda_\infty=\left(\begin{smallmatrix}0 & 1 \\ 0 & 0\end{smallmatrix}\right)$. 
For $\ell \in \mathrm{Iso}(V)$ we pick $\sigma_{\ell} \in\SL_2(\Z)$ 
such that $\sigma_{\ell}\ell_\infty=\ell$. 
By $\alpha_\ell$ we denote the width of the cusp $\ell$.

Heegner points are given as follows.
For $\lambda\in V(\Q)$ with $Q(\lambda)>0$ we let
\[
D_{\lambda}= \mathrm{span}(\lambda) \in D.
\]
For $Q(\lambda) \leq 0$ we set $D_{\lambda}=\emptyset$.
We denote the image of $D_{\lambda}$ in $M$ by $Z(\lambda)$.

\subsection{A lattice related to $\mathbf{\G_0(N)}$}\label{sec:lattice}

Following Bruinier and Ono \cite{BrOno,BrOno2}, we consider the lattice
\[
 L:=\left\{ \begin{pmatrix} b& -a/N \\ c&-b \end{pmatrix}; \quad a,b,c\in\Z \right\}.
\]
The dual lattice corresponding to the bilinear form $(\cdot,\cdot)$ is given by
\[
 L':=\left\{ \begin{pmatrix} b/2N& -a/N \\ c&-b/2N \end{pmatrix}; \quad a,b,c\in\Z \right\}.
\]
We identify the discriminant group $L'/L=:\dg$ 
with $\Z/2N\Z$, together with the $\Q/\Z$ valued quadratic form ${x \mapsto -x^2/4N}$.
The level of $L$ is $4N$. 

For a fundamental discriminant $\Delta\in\Z$ we will consider the rescaled lattice $\Delta L$ together with the quadratic form $Q_\Delta(\lambda):=\frac{Q(\lambda)}{\abs{\Delta}}$. The corresponding bilinear form is then given by $(\cdot,\cdot)_\Delta = \frac{1}{\abs{\Delta}} (\cdot,\cdot)$. The dual lattice of $\Delta L$ with respect to $(\cdot,\cdot)_\Delta$ is equal to $L'$. We denote the discriminant group $L'/\Delta L$ by $\dgdelta$. 

For $m \in \Q$ and $h \in \dg$, we let
\begin{equation*}
 L_{h,m}  = \left\{ \lambda \in L+h; Q(\lambda)=m  \right\}.
\end{equation*}
By reduction theory, if $m \neq 0$ the group $\G_0(N)$ acts on $ L_{h,m}$ with finitely many orbits.

We will also consider the one-dimensional lattice $K=\Z\left(\begin{smallmatrix} 1&0\\0&-1\end{smallmatrix}\right)\subset L$. We have $L=K+\Z\ell+\Z\ell'$ where $\ell$ and $\ell'$ are the primitive isotropic vectors
\[
 \ell=\begin{pmatrix} 0&1/N\\0&0\end{pmatrix},\quad\quad\quad\quad  \ell'=\begin{pmatrix} 0&0\\-1&0\end{pmatrix}.
\]
Then $K'/K\simeq L'/L$.


\subsection{The Weil representation and vector valued automorphic forms}\label{sec:weil}
By $\Mp_2(\Z)$ we denote the integral metaplectic group. It consists of pairs
$(\gamma, \phi)$, where $\gamma = {\smallabcd \in \SL_2(\Z)}$ and $\phi:\h\rightarrow \C$
is a holomorphic function with $\phi^2(\tau)=c\tau+d$.
The group $\tilde{\G}=\Mp_2(\Z)$ is generated by $S=(\smallSmatrix,\sqrt{\tau})$ and $T=(\smallTmatrix, 1)$. We let $\tilde{\G}_\infty:=\langle T\rangle \subset \tilde{\G}$.
We consider the Weil representation $\rho_\Delta$ of $\Mp_2(\Z)$ 
corresponding to the discriminant group $\dgdelta$ on the group ring $\C[\dgdelta]$,
equipped with the standard scalar product $\langle \cdot , \cdot \rangle$, conjugate-linear
in the second variable. We simply write $\rho$ for $\rho_1$.
 
Let $e(ma):=e^{2\pi i a}$. We write $\e_\delta$ for the standard basis element of $\C[\dgdelta]$ corresponding to $\delta \in \dgdelta$.
The action of $\rho_\Delta$ on basis vectors of $\C[\dgdelta]$ 
is given by the following formulas for 
the generators $S$ and $T$ of $\Mp_2(\Z)$
\begin{equation*}
 \rho_\Delta(T) \e_\delta = e(Q_\Delta(\delta)) \e_\delta,
\end{equation*}
and
\begin{equation*}
 \rho_\Delta(S) \e_\delta = \frac{\sqrt{i}}{\sqrt{\abs{\dgdelta}}}
				\sum_{\delta' \in \dgdelta} e(-(\delta',\delta)_\Delta) \e_{\delta'}.
\end{equation*}
Let $k \in \frac{1}{2}\Z$, and let $A_{k,\rho_\Delta}$ be the vector space of functions 
$f: \h \rightarrow \C[\dgdelta]$, such that for $(\gamma,\phi) \in \Mp_2(\Z)$ we have 
\begin{equation*}
	f(\gamma \tau) = \phi(\tau)^{2k} \rho_\Delta(\gamma, \phi) f(\tau).
\end{equation*}
A twice continuously differentiable function $f\in A_{k,\rho_\Delta}$ 
is called a \textit{(harmonic) weak Maa\ss{}  form of weight $k$ with respect to the representation $\rho_\Delta$} 
if it satisfies in addition:
\begin{enumerate}
 \item $\Delta_k f=0$,
 \item there is a $C>0$ such that $f(\tau)=O(e^{Cv}) $ as $v \rightarrow \infty$.
 \end{enumerate}
Here, we write $\tau=u+iv$ with $u,v \in \R$, and
$\Delta_k=-v^2\left(\frac{\partial^2}{\partial u^2}+\frac{\partial^2}{\partial v^2}\right)
+ikv\left(\frac{\partial}{\partial u}+i\frac{\partial}{\partial v}\right)$ 
is the weight $k$ Laplace operator.
We denote the space of such functions by $H_{k,\rho_\Delta}$. 
Moreover, we let $H^{+}_{k,\rho_\Delta}$ be the subspace of functions in $H_{k,\rho_\Delta}$ whose singularity at $\infty$ is locally given by the pole of a meromorphic function.
By $M^{\text{!}}_{k,\rho_\Delta} \subset H^{+}_{k,\rho_\Delta}$ 
we denote the subspace of weakly holomorphic modular forms.
 
Similarly, we can define scalar valued analogs of these spaces of automorphic forms.
In this case, we require analogous conditions at all cusps of $\G_0(N)$ in $(ii)$. We denote these spaces by
$H^{+}_{k}(N)$ and $M_k^{\text{!}}(N)$.

Note that the Fourier expansion of any harmonic weak Maa\ss{}  form uniquely decomposes into a holomorphic and a non-holomorphic part \cite[Section 3]{BrFu04}
\begin{align*}
& f^{+}=\sum\limits_{h\in 
L'/L}\sum\limits_{\substack{n\in\Q\\n\gg -\infty}}c^{+}(n,h)q^n \mathfrak{e}_h
\\
& f^{-}=\sum\limits_{h\in L'/L} \sum\limits_{n\in\Q }c^{-}(n,h)\G(1-k,4\pi \abs{n}v)q^n\mathfrak{e}_h,
\end{align*}
where $\G(a,x)$ denotes the incomplete $\G$-function. The first summand is called the holomorphic part of $f$, the second one the non-holomorphic part.

\subsection{Twisting vector valued modular forms}\label{sec:twistvec}
We define a generalized genus character for 
$\delta = \left(\begin{smallmatrix} b/2N& -a/N \\ c&-b/2N \end{smallmatrix}\right) \in L'$. 
From now on let $\Delta\in\Z$ be a fundamental discriminant and $r\in\Z$ such that $\Delta \equiv r^2 \ (\text{mod } 4N)$.

Then
\begin{equation*}
\chi_{\Delta}(\delta)=\chi_{\Delta}(\left[a,b,Nc\right]):=
\begin{cases}
\Deltaover{n}, & \text{if } \Delta | b^2-4Nac \text{ and } (b^2-4Nac)/\Delta \text{ is a}
\\
& \text{square mod } 4N \text{ and } \gcd(a,b,c,\Delta)=1,
\\
0, &\text{otherwise}.
\end{cases}
\end{equation*}
Here, $\left[a,b,Nc\right]$ is the integral binary quadratic form 
corresponding to $\delta$, and $n$ is any integer prime to $\Delta$ represented by $\left[a,b,Nc\right]$.

The function $\chi_{\Delta}$ is invariant under the action of $\G_0(N)$ and under the action of all Atkin-Lehner involutions.
It can be computed by the following formula \cite[Section I.2, Proposition 1]{GKZ}: If $\Delta=\Delta_1\Delta_2$ is a factorization of $\Delta$ into discriminants and $N=N_1N_2$ is a factorization of $N$ into positive factors such that $(\Delta_1,N_1a)=(\Delta_2,N_2c)=1$, then
\begin{equation*}
 \chi_{\Delta}(\left[a,b,Nc\right])=\left(\frac{\Delta_1}{N_1a}\right)\left(\frac{
\Delta_2}{N_2c}\right).
\end{equation*}
If no such factorizations of $\Delta$ and $N$ exist, we have $\chi_{\Delta}(\left[a,b,Nc\right])=0$.

Since $\chi_{\Delta}(\delta)$ depends only on $\delta \in L'$ modulo $\Delta L$, we can view it as a function on the discriminant group $\dgdelta$.

In \cite{AE} it was shown that we obtain an intertwiner of the Weil representations
corresponding to $\dg$ and $\dgdelta$ via $\chi_\Delta$.

\begin{proposition}[Proposition 3.2,\cite{AE}] \label{prop:intertwiner}
 We denote by $\pi: \dgdelta \rightarrow \dg$ the canonical projection.
 For $h \in \dg$, we define
 \begin{equation}
  \psi_{\Delta,r}(\e_h) := \sum_{\substack{\delta \in \dgdelta \\ \pi(\delta)=rh \\ Q_\Delta(\delta) \equiv \sgn(\Delta) Q(h) \, (\Z)}} \chi_\Delta(\delta) \e_\delta.
 \end{equation}
Then $\psi_{\Delta,r}: \dg \rightarrow \dgdelta$ defines an intertwining linear map between the representations $\widetilde{\rho}$ and $\rho_\Delta$, where
\[
  \widetilde{\rho} =
 \begin{cases}
 \rho & \text{if } \Delta>0, \\
		      \bar\rho & \text{if } \Delta<0.
 \end{cases}
\]
\end{proposition}

\begin{remark}
For a function $f \in A_{k,\rho_\Delta}$ this directly implies that the function $g: \h \rightarrow \C[\dg]$, ${g=\sum_{h \in \dg} g_h \e_h}$
 with $  {g_h := \left\langle \psi_{\Delta,r}(\e_h), f \right\rangle}$, 
 is contained in $A_{k,\widetilde{\rho}}$.
\end{remark}

\subsection{Differential Operators}\label{sec:diffop}
We define the Maa\ss{}  lowering and raising operator by
\[
L_k=-2iv^2\frac{\partial}{\partial \bar{\tau}} \quad\quad\quad\mathrm{and}
\quad\quad\quad R_k=2i\frac{\partial}{\partial\tau}+kv^{-1}.
\]
The lowering operator $L_k$ takes automorphic forms of weight $k$ to automorphic forms of weight $k-2$ and the raising operator $R_k$ takes automorphic forms of weight $k$ to automorphic forms of weight $k+2$. Moreover, these operators commute with the slash operator
and they satisfy the following relations with the weighted Laplace operator

\begin{equation}\label{eq:RkDk}
 R_k\Delta_k=(\Delta_{k+2}-k)R_k,
\end{equation}
\begin{equation}\label{eq:LkDk}
 \Delta_{k-2}L_k=L_k(\Delta_k+2-k).
\end{equation}
We also define iterated versions of the raising and lowering operators
\[
 R_k^n=R_{k+2(n-1)}\circ \cdots\circ R_{k+2} \circ R_{k},
\quad\quad
 L_k^n=L_{k-2(n-1)}\circ\cdots L_{k-2}\circ L_{k}.
\]
For $n=0$ we set $R_k^0=L_k^0=\text{id}$.

Using (\ref{eq:RkDk}) and (\ref{eq:LkDk}) we can show by induction that these operators commute with the weighted Laplacian.
\begin{lemma}\label{lm:reldiff}
We have
\begin{align*}
 \Delta_0 R_{-2k}^k & =R_{-2k}^k\left(\Delta_{-2k}-k(k+1)\right)
 \\
 \Delta_{3/2+k}R_{3/2}^{k/2}&=R_{3/2}^{k/2}\left(\Delta_{3/2}+\frac{k}{4}(k+1)\right)
 \\ 
 \Delta_{1/2-k}L_{3/2}^{(k+1)/2}&=L_{3/2}^{(k+1)/2}\left(\Delta_{3/2}+\frac{k}{4}(k+1)\right).
\end{align*}
\end{lemma}

We let $\xi_k=v^{k-2}\overline{L_kf(\tau)}=R_{-k}v^k\overline{f(\tau)}$. Then by Proposition 3.2 of \cite{BrFu04}
\[
 \xi_k: H_{k,\rho}\rightarrow M^{\text{!}}_{2-k,\bar{\rho}}
\]
and
\[
 \xi_k: H^+_{k,\rho}\rightarrow S_{2-k,\bar{\rho}}.
\]

In \cite{BrFu04} Bruinier and Funke define a bilinear pairing between 
the spaces $M_{2-k,\bar{\rho}}$ and $H^+_{k,\rho}$ by
\[
\left\{g,f\right\}=\left(g,\xi_k(f)\right)_{2-k,\bar{\rho}},
\]
where $g\in M_{2-k,\bar{\rho}}$, $f\in H^+_{k,\rho}$, and $(\cdot,\cdot)$ denotes the Petersson scalar product. They obtain a duality result for $g\in M_{2-k,\bar{\rho}}$ and $f\in H^+_{k,\rho}$ (see Proposition 3.5 of \cite{BrFu04}). This pairing extends to weakly holomorphic forms $g\in M^{\text{!}}_{2-k,\bar{\rho}}$. 

\begin{proposition}\label{prop:extpairing}
 For $f\in H^+_{k,\rho}$ and $g\in M^{\text{!}}_{2-k,\bar{\rho}}$ we have
\[
 \left\{g,f\right\}=\sum\limits_{h\in L'/L}\sum\limits_{ n\in\Q} a^{+}(n,h)b(- n,h).
\]
Here $a^{+}(n,h)$ denotes the $(n,h)$-th coefficient of the holomorphic part of $f$ and $b(n,h)$ the $(n,h)$-th coefficient of $g$.
\end{proposition}


\subsection{Poincar\'{e} series and Whittaker functions}\label{sec:pc}
We recall some facts on Poincar\'{e} series with exponential growth at the cusps following Section 2.6 of \cite{BrOno2}.

We let $k\in\frac12\Z$, and $M_{\nu,\mu}(z)$ and $W_{\nu,\mu}(z)$ denote the usual Whittaker functions (see p. 190 of \cite{Pocket}). For $s\in\C$ and $y\in\R_{>0}$ we put
\[
\mathcal{M}_{s,k}(y)=y^{-k/2}M_{-\frac{k}{2},s-\frac12}(y).
\]
We let $\G_\infty$ be the subgroup of $\G_0(N)$ generated by $\left(\begin{smallmatrix}1&1\\0&1\end{smallmatrix}\right)$. For $k\in\Z$, $m\in\N$, $z=x+iy\in\h$ and $s\in \C$ with $\Re(s)>1$, we define
\begin{equation}\label{def:poincare}
 F_m(z,s,k)=\frac{1}{2\G(2s)}\sum\limits_{\gamma\in\G_\infty\setminus\G_0(N)}\left[\M_{s,k}(4\pi my)e(-mx)\right]|_k\ \gamma.
\end{equation}
This Poincar\'{e} series converges for $\Re(s)>1$, and it is an eigenfunction of $\Delta_k$ with eigenvalue $s(1-s)+(k^2-2k)/4$. Its specialization at $s_0=1-k/2$ is a harmonic Maa\ss{}  form \cite[Proposition 1.10]{BrHabil}. The principal part at the cusp $\infty$ is given by $q^{-m}+C$ for some constant $C\in\C$. The principal parts at the other cusps are constant.

The Poincar\'{e} series behave nicely under the Maa\ss{}  raising and lowering operator.
\begin{proposition}[Proposition 2.2, \cite{BrOno2}]\label{pro:pcmaassrai}
We have that
\[
 R_kF_m(z,s,k)=4\pi m \left(s+\frac{k}{2}\right)F_m(z,s,k+2).
\]
\end{proposition}
Similarly
\begin{proposition}\label{pro:pcmaasslow}
We have that
\[
L_kF_m(z,s,k)=\frac{1}{4\pi m} \left(s-\frac{k}{2}\right) F_m(z,s,k-2).
\]
\end{proposition}
\begin{proof}
This is proven analogously to Proposition 2.2 in \cite{BrOno2}. Now we use equations (13.4.11) and (13.1.32) in \cite{Pocket}.
\end{proof}

We now define $\C[L'/L]$-valued analogs of these series. Let $h\in L'/L$ and $m\in\Z-Q(h)$ be positive. For $k\in\left(\Z-\frac12\right)_{< 0}$ we let
\[
\F_{m,h}(\tau,s,k)=\frac{1}{2\G(2s)}\sum\limits_{\gamma\in \widetilde{\G}_\infty\setminus\widetilde{\G}}\left. \left[\M_{s,k}(4\pi m y)e(-mx)\mathfrak{e}_h\right]\right|_{k,\rho}\ \gamma,
\]
and for $k\in\left(\Z-\frac12\right)_{\geq 0}$ we let
\[
\F_{m,h}(\tau,s,k)=\frac{1}{2}\sum\limits_{\gamma\in \widetilde{\G}_\infty\setminus\widetilde{\G}}\left. \left[\M_{s,k}(4\pi m y)e(-mx)\mathfrak{e}_h\right]\right|_{k,\rho}\ \gamma.
\]
The series $\F_{m,h}(\tau,s,k)$ converges for $\Re(s)>1$ and it defines a weak Maa\ss{}  form of weight $k$ for $\widetilde{\G}$ with representation $\rho$. The special value at $s=1-k/2$ is harmonic \cite[Proposition 1.10]{BrHabil}. For $k\in\Z-\frac12$ the principal part is given by $q^{-m}\mathfrak{e}_h+q^{-m}\mathfrak{e}_{-h}+C$ for some constant $C\in\C[L'/L]$.

\begin{remark}
 For $k<0$ these Poincar\'{e} series series span the space $H^+_{k,\rho}$ \cite[Proposition 1.12]{BrHabil}.
\end{remark}

Later we will also consider the $W$-Whittaker function
\[
 \mathcal{W}_{s,k}(y)=y^{-k/2} W_{k/2,s-1/2}(y), \quad y>0.
\]
We compute its behavior under the Maa\ss{}  raising and lowering operator.
\begin{proposition}\label{prop:diffwhittw}
 For $m>0$ and $y>0$ we have that
\begin{equation*}
L_k \mathcal{W}_{s,k}(4\pi m y)e(mx)= \frac{1}{4\pi m} \left(s-\frac{k}{2}\right)\left(1-s-\frac{k}{2}\right)\mathcal{W}_{s,k-2}(4\pi m y)e(mx)
\end{equation*}
and
\begin{equation*}
 R_k \mathcal{W}_{s,k}(4\pi m y)e(mx)= -4\pi m \mathcal{W}_{s,k+2}(4\pi m y)e(mx).
\end{equation*}
\end{proposition}
\begin{proof}
 For the first equation we use (13.1.33) and (13.4.23) and for the second one (13.1.33) and (13.4.26) in \cite{Pocket}.
\end{proof}

\section{Theta Functions}\label{sec:theta}

\subsection{The Kudla-Millson theta function}\label{sec:thetakm}
We let 
\[
 \varphi^0_{\Delta,\text{KM}}(\lambda,z)
 =\left(\frac{1}{\abs{\Delta}}(\lambda,\lambda(z))^2-\frac{1}{2\pi} \right) e^{-2\pi R(\lambda,z)/|\Delta|}\Omega,
\]
where 
$R(\lambda,z):=\frac{1}{2}(\lambda,\lambda(z))^2-(\lambda,\lambda)$ 
and $\Omega = \frac{i}{2}\frac{dz \wedge d\bar{z}}{y^2}$. It is a Schwartz function constructed by Kudla and Millson \cite{KM86}. From now on, if $\Delta=1$ we omit the index $\Delta$ and simply write $\varphi^0_{\text{KM}}(\lambda,z)$ .

For  $\delta \in \dgdelta$ we define a theta function $\Theta_{\delta}(\tau,z)$ for 
$\tau, z \in \h$ as in \cite{BrFu06,AE}. Let 
 $\varphi_{\text{KM}}(\lambda,\tau,z)=e^{2\pi i Q_\Delta(\lambda)\tau} \varphi^0_{\Delta,\text{KM}}(\sqrt{v}\lambda,z)$ and
\begin{equation*}
	\Theta_{\delta}(\tau,z,\phikm)=\sum\limits_{\lambda\in \Delta L+\delta} \varphi_{\text{KM}} (\lambda,\tau, z).
\end{equation*}
The vector valued theta series
\[
\Theta_{\dgdelta}(\tau,z,\phikm)=\sum\limits_{\delta\in \dgdelta} \Theta_{\delta}(\tau,z,\phikm) \mathfrak{e}_{\delta}
\]
is then a $C^{\infty}$-automorphic form of weight $3/2$ which transforms with respect to the representation $\rho_\Delta$ \cite{BrFu06}.

We obtain a  $\C[\dg]$-valued twisted theta function by setting
\begin{equation*}
	\thetaL{\phikm} := \sum_{h \in \dg} \Theta_{\Delta,r,h}(\tau,z,\phikm) \e_h,
\end{equation*}
where the corresponding components for $h \in \dg$ are defined as
\begin{equation*}
\begin{split}
 \Theta_{\Delta,r,h}(\tau,z,\phikm) &= \left\langle \psi_{\Delta,r}(\e_h), \overline{\Theta_{\dgdelta}(\tau,z,\phikm)} \right\rangle
\\
& =
  \sum\limits_{\substack{\delta\in \dgdelta\\ \pi(\delta) = rh \\Q_\Delta(\delta)\equiv\sgn(\Delta)Q(h)\, (\Z)}} \chi_{\Delta}(\delta)\Theta_{\delta}(\tau,z,\phikm).
\end{split}
\end{equation*}

By Proposition \ref{prop:intertwiner} the theta function $\thetaL{\phikm}$ is a non-holomorphic $\C[\dg]$-valued modular form of weight
$3/2$ for the representation $\widetilde{\rho}$. 
Furthermore, it is a non-holomorphic automorphic form of weight 0 for $\G_0(N)$ in the variable $z \in D$.

Following \cite{BrFu06} we rewrite the Kudla-Millson kernel as a Poincar\'{e} series using the smaller lattice $K$.
We let $\epsilon =1$, when $\Delta>0$, and $\epsilon =i$, when $\Delta<0$. The following proposition can be found in \cite{Ehlen}.

\begin{proposition}\label{twistsmallK}
We have
\begin{align*}
 & \thetaL{\phikm}= -y\,\frac{N^{3/2}}{2\abs{\Delta}}\, \bar{\epsilon}\,\sum\limits_{n=1}^{\infty}n^2\sum_{\gamma\in \widetilde{\G}_\infty\setminus\widetilde{\G}}  \left(\frac{\Delta}{n}\right) 
\\
& \times\left[\exp\left(-\pi\frac{y^2Nn^2}{v\abs{\Delta}}\right) v^{-3/2} \sum_{\lambda \in  K'} e\left(\abs{\Delta}Q(\lambda)\bar{\tau}-2N\lambda n x \right)\mathfrak{e}_{r\lambda}\right]\Bigg|_{3/2,\widetilde{\rho}_K}\gamma \, dx dy .
\end{align*}
\end{proposition}

\subsection{The Siegel theta function}\label{sec:thetasiegel}

We define the Siegel theta function by
\begin{equation*}
\thetaL{\phis}:=v\sum\limits_{\lambda\in L'} e^{-2\pi v R(\lambda,z)/\abs{\Delta}}e(Q_{\Delta}(\lambda)\tau)\mathfrak{e}_\lambda.
\end{equation*}
Here $
 \varphi_{\text{S}}(\lambda,z)=e^{-\pi (\lambda,\lambda)_{\Delta,z}}
$
is the Gaussian on $V(\R)$ associated to the majorant $(\cdot,\cdot)_z$. 
The Siegel theta function is a $\G_0(N)$-invariant function in $z$ and a non-holomorphic modular form of weight $-1/2$ for $\widetilde{\G}$ with representation $\widetilde{\rho}$.

By $\partial, \bar{\partial}$ we denote the usual differentials on $\mathbb{D}$. We set $d^{\text{c}}=\frac{1}{4\pi i}(\partial-\bar{\partial})$, so that $dd^{\text{c}}=-\frac{1}{2\pi i}\partial\bar{\partial}$. The Kudla-Millson theta function and the Siegel theta function are related by the identity \cite[Theorem 4.4]{BrFu04} 
\begin{equation}\label{eq:DeltaTheta}
L_{3/2,\tau}\thetaL{\phikm}=-dd^{\text{c}}\thetaL{\phis}=\frac{1}{4\pi}\Delta_{0,z}\thetaL{\phis}\cdot\Omega.
\end{equation}
For the Kudla-Millson theta kernel we have \cite[Equation (2.17)]{BrOno2}
\begin{equation}\label{eq:deltatheta}
\Delta_{3/2,\tau} \thetaL{\phikm}= \frac14 \Delta_{0,z}\thetaL{\phikm}.
\end{equation}


\section{The Theta Lift}\label{sec:thetalift}

Let $f$ be a harmonic weak Maa\ss{}  form in $H^{+}_{-2k}(N)$. Following Bruinier and Ono \cite{BrOno2} we define a theta lift as follows. For even $k$ we let
\[
 \LE=R_{3/2,\tau}^{k/2} \int_M (R^{k}_{-2k,z}f)(z)\thetaL{\phikm}
\]
and for $k$ odd
\[
 \LO=L_{3/2,\tau}^{(k+1)/2} \int_M (R^{k}_{-2k,z}f)(z)\thetaL{\phikm}.
\]
We drop the index $\Delta,r$ if we consider the untwisted lift, and the superscript ``e'' respectively ``o'' in the case that the same proof works for even and odd $k$.
We use a subscript $h$ to denote the components of the lift.

Note that the transformation properties of the twisted Kudla-Millson theta function directly imply that the lift transforms with representation $\widetilde{\rho}$.

Recall that $O(L'/L)$ can be identified with the group generated by the Atkin-Lehner involutions. It acts on harmonic Maa\ss{}  forms by the slash operator and on $\C[L'/L]$-valued modular forms with respect to the Weil representation $\widetilde{\rho}$ through the natural action on $\C[L'/L]$. By a straightforward computation we see that the theta lift is equivariant with respect to the action of $O(L'/L)$ (compare \cite[Proposition 3.1]{BrOno2}).
\begin{proposition}\label{prop:actionO}
 For $\gamma\in O(L'/L)$ and $h\in L'/L$, we have
\[
 \Lambda_{\Delta,r,\gamma h}(\tau,f)=\Lambda_{\Delta,r,h}(\tau, \left. f\right|_{-2k}\gamma^{-1}).
\]
\end{proposition}

\begin{proposition}\label{prop:efodd}
 Let $f$ be an eigenform of $\Delta_{-2k,z}$ with eigenvalue $\lambda$. Then $\LO$ is an eigenform of $\Delta_{1/2-k,\tau}$ with eigenvalue $\frac{\lambda}{4}$, and $\LE$ is an eigenform of $\Delta_{3/2+k,\tau}$ with eigenvalue $\frac{\lambda}{4}$.
\end{proposition}
\begin{proof}
 Using Lemma \ref{lm:reldiff} we see that $\Delta_{1/2-k,\tau}\LO$ equals
 \begin{align*}
 L_{3/2,\tau}^{\frac{k+1}{2}} \int_M (R^{k}_{-2k,z}f)(z)\Delta_{3/2,\tau}\thetaL{\phikm}+\frac{k}{4}(k+1) \LO.
 \end{align*}
Via the relation between the two theta kernels \eqref{eq:DeltaTheta} we obtain
 \begin{align}\label{eq:ev1}
 \frac14 L_{3/2,\tau}^{\frac{k+1}{2}} \int_M (R^{k}_{-2k,z}f)(z)\Delta_{0,z}\thetaL{\phikm}+\frac{k}{4}(k+1)  \LO.
 \end{align}
By the rapid decay of the Kudla-Millson theta function \cite[Proposition 4.1]{Funke} we may move the Laplacian. Using  Lemma \ref{lm:reldiff} we then obtain that (\ref{eq:ev1}) equals
\begin{align*}
  \frac14 L_{3/2,\tau}^{\frac{k+1}{2}} \int_M (R_{-2k,z}^{k}\Delta_{-2k,z}f)(z)\thetaL{\phikm}.
\end{align*}
Since $f$ is an eigenform with eigenvalue $\lambda$ this equals $\frac{\lambda}{4}\LO$.
For even $k$ we argue analogously.

\end{proof}


We compute the lift of the Poincar\'{e} series $ F_m(z,s,-2k)$.
\begin{theorem}\label{thm:liftpoincare}
For $k\geq 0$ we have
\begin{align*}
&\Lambda^{\text{e}}_{\Delta,r}(\tau, F_m(z,s,-2k))
\\
&\quad\quad\quad=C^{\text{e}}\cdot\sum\limits_{n\mid m}\left(\frac{\Delta}{n}\right) n^{-(k+1)} \F_{\frac{m^2}{4Nn^2}\abs{\Delta},-\frac{m}{n}r}\left(\tau,\frac{s}{2}+\frac14,\frac32+k\right),
\end{align*}
where
\begin{align*}
C^{\text{e}}& =-\frac{2^{2s+2k-1} m^{2k+1} \pi^{(3k-1)/2}\abs{\Delta}^{(k+1)/2} \bar{\epsilon}}{N^{k/2}} \, \frac{\G\left(\frac{s}{2}+1\right)}{\G(2s)}
\\
&\quad\quad\quad\quad\quad\quad\quad\quad\times \prod_{j=0}^{k-1}(s+j-k)\prod_{j=0}^{k/2-1}\left(\frac{s}{2}+1+j\right)
\end{align*}
and
\begin{align*}
\Lambda^{\text{o}}_{\Delta,r}(\tau, F_m(z,s,-2k))= C^{\text{o}}\cdot
\sum\limits_{n\mid m}\left(\frac{\Delta}{n}\right) n^{k} \F_{\frac{m^2}{4Nn^2}\abs{\Delta},-\frac{m}{n}r}\left(\tau,\frac{s}{2}+\frac14,\frac12-k\right),
\end{align*}
where 
\begin{align*}
C^{\text{o}}=-\frac{2^{2k-s}\abs{\Delta}^{-k/2}\bar{\epsilon}}{\G\left(\frac{s}{2}+\frac12\right)} N^{(k+1)/2}\pi^{k/2}s \prod_{j=0}^{k-1}(s+j-k)\prod_{j=0}^{(k-1)/2}\left(\frac{s}{2}-\frac12-j\right).
\end{align*}
\end{theorem}

\begin{proof}
For the explicit evaluation of the lift of Poincar\'{e} series we generalize the proof of Bruinier and Ono \cite{BrOno2}.
Repeatedly applying Proposition \ref{pro:pcmaassrai} implies (by induction)
\begin{align}
 & \label{proof:pc11} \int_M (R^{k}_{-2k,z}F_m(z,s,-2k))\thetaL{\phikm}
\\
&\notag=
 (4\pi m)^k\prod_{j=0}^{k-1}(s+j-k) \int_M F_m(z,s,0)\thetaL{\phikm}.
\end{align}
Using the definition of the Poincar\'{e} series (\ref{def:poincare}) and an unfolding argument we obtain
\begin{align*}
 \frac{1}{\G(2s)} (4\pi m)^k\prod_{j=0}^{k-1}(s+j-k) \int_{\G_\infty\setminus\h}\M_{s,0}(4\pi m y)e(-mx)\thetaL{\phikm}.
\end{align*}
By Lemma \ref{twistsmallK} this equals
\begin{align*}
 -\frac{N^{3/2}\bar{\epsilon}}{2\abs{\Delta}} \frac{(4\pi m)^k}{\G(2s)} \prod_{j=0}^{k-1}(s+j-k)  \sum\limits_{n=1}^\infty \left(\frac{\Delta}{n}\right)n^2\sum\limits_{\gamma\in\widetilde{\G}_\infty\setminus\widetilde{\G}} I(\tau,s,m,n)|_{3/2,\widetilde{\rho}_K} \ \gamma,
\end{align*}
where
\begin{align*}
  I(\tau,s,m,n)&= \int_{y=0}^\infty\int_{x=0}^1 y \M_{s,0}(4\pi m y)e(-mx) \exp\left(-\frac{\pi n^2Ny^2}{\abs{\Delta}v}\right)
  \\
  & \quad\quad\quad\quad\quad\times v^{-3/2}\sum_ {\lambda\in K'}e\left(\abs{\Delta}Q(\lambda)\bar{\tau}-2N\lambda n x\right)\mathfrak{e}_{r\lambda}dxdy.
\end{align*}
Identifying $K'=\Z\left(\begin{smallmatrix} 1/2N&0\\0&-1/2N\end{smallmatrix}\right)$ we find that 
\[
 \sum_ {\lambda\in K'}e\left(\abs{\Delta}Q(\lambda)\bar{\tau}-2N\lambda n x\right)\mathfrak{e}_{r\lambda}=\sum\limits_{b\in\Z}e\left(-\abs{\Delta}\frac{b^2}{4N}\bar{\tau}-nbx\right)\mathfrak{e}_{rb}.
\]
Inserting this in the formula for $I(\tau, s,m,n)$, and integrating over $x$, we see that $I(\tau,s,m,n)$ vanishes whenever $n\nmid m$ and the only summand occurs for $b=-m/n$, when $n\mid m$. Thus, $ I(\tau,s,m,n)$ equals
\begin{align}\label{proof:pc222}
v^{-3/2} e\left(-\abs{\Delta}\frac{m^2}{4Nn^2}\bar{\tau}\right)\ \cdot\ \int_{y=0}^\infty y\M_{s,0}(4\pi m y) \exp\left(-\frac{\pi n^2Ny^2}{\abs{\Delta}v}\right)dy  \,\, \mathfrak{e}_{-rm/n}.
 \end{align}
To evaluate the integral in \eqref{proof:pc222} note that (see for example (13.6.3) in \cite{Pocket})
\[
\M_{s,0}(4\pi m y)=2^{2s-1}\G\left(s+\frac12\right)\sqrt{4\pi m y}\cdot I_{s-1/2}(2\pi m y).
\]
Substituting $t=y^2$ yields
\begin{align*}
& \int_{y=0}^\infty y\M_{s,0}(4\pi m y) \exp\left(-\frac{\pi n^2Ny^2}{\abs{\Delta}v}\right)dy
\\
&\notag = 2^{2s-1}\G\left(s+\frac12\right) \int_{y=0}^\infty y \sqrt{4\pi m y} \ I_{s-1/2}(2\pi m y)  \exp\left(-\frac{\pi n^2Ny^2}{\abs{\Delta}v}\right)dy
\\
&\notag = 2^{2s-1}\G\left(s+\frac12\right)  \sqrt{m\pi} \int_{t=0}^\infty t^{1/4}  I_{s-1/2}(2\pi m t^{1/2})  \exp\left(-\frac{\pi n^2Nt}{\abs{\Delta}v}\right)dt.
\end{align*}
The last integral is a Laplace transform and is computed in \cite{tables} (see (20) on p. 197). It equals
\[
\frac{\G\left(\frac{s}{2}+1\right)}{\G\left(s+\frac12\right)}(\pi m)^{-1} \left(\frac{\pi n^2N}{\abs{\Delta}v}\right)^{-3/4} \exp\left(\frac{\pi m^2\abs{\Delta}v}{2n^2N}\right)  M_{-\frac{3}{4}, \frac{s}{2}-\frac{1}{4}}\left(\frac{\pi m^2\abs{\Delta}v}{n^2N}\right).
\]
Inserting this we obtain that 
\begin{align*}
& I(\tau,s,m,n)= 2^{2s-1} \G\left(\frac{s}{2}+1\right) (\pi m)^{-2} \left(\frac{\pi m^2\abs{\Delta}}{n^2N}\right)^{3/2} 
\\
&\quad\quad\quad\quad\quad\quad\quad\times\M_{s/2+1/4,3/2}\left(\frac{\pi m^2\abs{\Delta}v}{n^2N}\right)e\left(-\frac{m^2\abs{\Delta}u}{4n^2N}\right)\mathfrak{e}_{-rm/n}.
\end{align*}
Therefore, we have that \eqref{proof:pc11} equals
\begin{equation}\label{proof:pc2222}
 -\frac{N^{3/2}}{2\abs{\Delta}}\bar{\epsilon}\ \frac{1}{\G(2s)} (4\pi m)^k\prod_{j=0}^{k-1}(s+j-k)  \sum\limits_{n=1}^\infty\left(\frac{\Delta}{n}\right) n^2\sum\limits_{\gamma\in\widetilde{\G}_\infty\setminus\widetilde{\G}} I(\tau,s,m,n)|_{3/2,\widetilde{\rho}_K}\gamma.
\end{equation}
For $k=0$ and even $k$ equation \eqref{proof:pc2222} can be rewritten as
\begin{align*}
& C\cdot \frac{1}{2} \sum\limits_{n=1}^\infty \left(\frac{\Delta}{n}\right)n^{-1}
 \\
&\quad\quad\times\sum\limits_{\gamma\in\widetilde{\G}_\infty\setminus\widetilde{\G}} 
 \left.\left[ \M_{\frac{s}{2}+\frac{1}{4},\frac{3}{2}}\left(\frac{\pi m^2\abs{\Delta}v}{n^2N}\right)e\left(-\frac{ m^2\abs{\Delta}u}{4n^2N}\right)\mathfrak{e}_{-rm/n}\right] \right|_{3/2,\widetilde{\rho}_K} \ \gamma
\\
&=C\cdot\sum\limits_{n=1}^\infty \left(\frac{\Delta}{n}\right) n^{-1} \mathcal{F}_{\frac{m^2}{4Nn^2}\abs{\Delta},-\frac{m}{n}r}\left(\tau,\frac{s}{2}+\frac14,\frac32\right),
\end{align*}
where
\[
 C=- 2^{2s+2k-1}\, m^{k+1}\, \pi^{k-1/2} \sqrt{\abs{\Delta}}\, \bar{\epsilon}\, \frac{\G\left(\frac{s}{2}+1\right)}{\G(2s)}\prod_{j=0}^{k-1}(s+j-k).
\]
If $k\neq0$ we consider $R_{3/2,\tau}^{k/2}$ of this expression. By the commutativity of the raising and the slash operator Proposition \ref{pro:pcmaassrai} implies
\begin{align*}
&R_{3/2,\tau}^{k/2}\mathcal{F}_{\frac{m^2}{4Nn^2}\abs{\Delta},-\frac{m}{n}r}\left(\tau,\frac{s}{2}+\frac14,\frac32\right)=\left(4\pi \frac{m^2\abs{\Delta}}{4Nn^2}\right)^{k/2}
\\
&\quad\quad\quad\times\prod\limits_{j=0}^{k/2-1}\left(\frac{s}{2}+\frac14+\frac{3/2+2j}{2}\right)\mathcal{F}_{\frac{m^2}{4Nn^2}\abs{\Delta},-\frac{m}{n}r}\left(\tau,\frac{s}{2}+\frac14,\frac32+k\right).
\end{align*}
We collect terms to obtain $C^{\text{e}}$ as in the statement of the theorem.

For odd $k$ we rewrite \eqref{proof:pc2222} as follows
\begin{align*}
& C'\cdot \frac{1}{2\G\left(s+\frac12\right)} \sum\limits_{n=1}^\infty \left(\frac{\Delta}{n}\right)n^{-1}
 \\
&\quad\quad\times\sum\limits_{\gamma\in\widetilde{\G}_\infty\setminus\widetilde{\G}} 
 \left.\left[ \M_{\frac{s}{2}+\frac{1}{4},\frac{3}{2}}\left(\frac{\pi m^2\abs{\Delta}v}{n^2N}\right)e\left(-\frac{ m^2\abs{\Delta}u}{4n^2N}\right)\mathfrak{e}_{-rm/n}\right] \right|_{3/2,\widetilde{\rho}_K} \ \gamma
\\
&=C'\cdot\sum\limits_{n=1}^\infty \left(\frac{\Delta}{n}\right) n^{-1} \mathcal{F}_{\frac{m^2}{4Nn^2}\abs{\Delta},-\frac{m}{n}r}\left(\tau,\frac{s}{2}+\frac14,\frac32\right),
\end{align*}
where 
\[
C':=-\frac{2^{2k-s}\sqrt{\abs{\Delta}} \,\bar{\epsilon}}{\G\left(\frac{s}{2}+\frac12\right)}\,m^{k+1}\,\pi^{k+1/2} s\,\prod_{j=0}^{k-1}(s+j-k).
\]
A repeated application of the lowering operator (Lemma \ref{pro:pcmaasslow}) yields the statement in the Theorem. 

\end{proof}

\begin{corollary}\label{cor:lifttrans}
Let $N$ be square-free and $k>0$.
If $f\in H^{+}_{-2k}(N)$ is a harmonic Maa\ss{}  form of weight $-2k$ for $\G_0(N)$, then $\LE$ belongs to $M^{\text{!}}_{3/2+k,\widetilde{\rho}}$ and $\LO$ belongs to $H^{+}_{1/2-k,\widetilde{\rho}}$. 
\end{corollary}
\begin{proof}
Note that for $f\in H^{+}_{3/2+k,\widetilde{\rho}}$, where $k>0$, we have $\xi_{3/2+k}(f)\in S_{1/2-k,\overline{\widetilde{\rho}}}$. Since $\mathrm{dim}(S_{1/2-k,\overline{\widetilde{\rho}}})=0$ for $k> 0$, this implies that $f\in M^{\text{!}}_{3/2+k,\widetilde{\rho}}$.

For odd $k$ the proof is similar to the proof of \cite[Corollary 3.4]{BrOno2}: For $m\in\N_{>0}$ the Poincar\'{e} series span the subspace $H^{+,\infty}_{-2k}(N)$ of harmonic Maa\ss{}  forms having only a pole at $\infty$. By Theorem \ref{thm:liftpoincare} we find that the image of $H^{+,\infty}_{-2k}(N)$ is contained in  $M^{\text{!}}_{3/2+k,\widetilde{\rho}}$ for even $k$, and $H^{+}_{1/2-k,\widetilde{\rho}}$ for odd $k$.
Moreover,
\[
 H^{+}_{-2k}(N)=\sum_{\gamma\in O(L'/L)} \gamma H^{+,\infty}_{-2k}(N),
\]
since the group $O(L'/L)$ of Atkin-Lehner involutions acts transitively on the cusps of $\G_0(N)$ for square-free $N$. Applying Proposition \ref{prop:actionO} now implies the result.
\end{proof}

\begin{remark}
  In the case $k=0$ one has to assume that the constant coefficient of the input function in $M^{\text{!}}_0$ vanishes to guarantee that the lift is weakly holomorphic. Otherwise it is a harmonic Maa\ss{}  form in $H_{3/2,\widetilde{\rho}}$. This case was treated in detail in \cite{BrFu06,AE}.
\end{remark}

For a cusp form $g=\sum_{n=1}^\infty a_n q^n$ we let $L(g,D,s)$ be its twisted $L$-function
\[
L(g,D,s)=\sum_{n=1}^{\infty} \left(\frac{D}{n}\right)a_n n^{-s}.
\]
We obtain a Kohnen-type theorem for the coefficients of the non-holomorphic part of the theta lift.

\begin{theorem}
If $k$ is odd, $N$ is square-free, and $f\in H^+_{-2k}(N)$, then $\LO$ is weakly holomorphic if and only if $f$ is weakly holomorphic or if we have
\[
L(\xi_{-2k}(f),\Delta, k+1) = 0.
\]
\end{theorem}
\begin{proof}
Here we partly follow the proof of Bruinier and Ono in \cite{BrOno2}. Since $O(L'/L)$ acts transitively on the cusps, it suffices to consider the case when $f$ has only a pole at $\infty$. Again, we obtain the result for the entire space $H_{-2k}^{+}(N)$ by using Proposition \ref{prop:actionO}. For $f\in H_{-2k}^{+,\infty}(N)$ we denote the Fourier expansion of the holomorphic part at the cusp $\infty$ by
\[
 f(z)=\sum_{m\in\Z} a_f(m)e(mz).
\]
Then we can write $f$ as a linear combination of Poincar\'{e} series 
\[
 f(z)=\sum_{m>0} a_f(-m) F_m(z,1+k,-2k).
\]
By Theorem \ref{thm:liftpoincare} the principal part of $\LO$ is given by
\[
 C^{\text{o}}\cdot \sum_{m>0} a_f(-m) \sum\limits_{n\mid m} \left(\frac{\Delta}{n}\right)n^{k} e\left(-\frac{m^2\abs{\Delta}}{4Nn^2}z\right)(\mathfrak{e}_{rm/n}+\mathfrak{e}_{-rm/n}),
\]
where $C^{\text{o}}$ is as in Theorem \ref{thm:liftpoincare}.

We now use well-known pairing between the spaces $H^{+}_{1/2-k,\widetilde{\rho}}$ and $S_{3/2+k,\overline{\widetilde{\rho}}}$ (see Section \ref{sec:diffop}). To prove that the lift is weakly holomorphic we have to show that $\left\{\LO,g\right\}=0$ for every cusp form $g\in S_{3/2+k,\overline{\widetilde{\rho}}}$. Denoting the Fourier coefficients of $g$ by $b(n,h)$, we have
\begin{align*}
 \left\{\LO,g\right\} &= 2C^{\text{o}} \sum_{m>0} a_f(-m) \sum\limits_{n\mid m} n^{k}  \left(\frac{\Delta}{n}\right) b\left(\frac{m^2\abs{\Delta}}{4Nn^2},\frac{m}{n}r\right)
\\
&= 2C^{\text{o}} \left\{f, \mathcal{S}_{\Delta, r}(g)\right\} = 2C^{\text{o}}(\xi_{-2k}(f),\mathcal{S}_{\Delta, r}(g)),
\end{align*}
where $\mathcal{S}_{\Delta,r}(g)\in S_{2k+2}(N)$ denotes the Shimura lift of $g$ as in \cite{skoruppa}. If $f$ is weakly holomorphic this expression vanishes, since $\xi_{-2k}(f)=0$.

If $f\in H^+_{-2k}(N)\setminus M^{\text{!}}_{-2k}(N)$, we have by the adjointness of the Shintani and Shimura lifting (see for example Section II.3 of \cite{GKZ}, and \cite{skoruppa,skoruppa2} for the case of skew-holomorphic Jacobi forms)
\[
   (\xi_{-2k}(f),\mathcal{S}_{\Delta, r}(g))= (\mathcal{S}^*_{\Delta, r}(\xi_{-2k}(f)),g),
\]
where $S^*_{\Delta,r}$ denotes the Shintani lifting. This equals zero for all cusp forms $g$ if and only if the  Shintani lifting of $\xi_{-2k}(f)$ vanishes. We have that (in terms of Jacobi forms; for the definition of Jacobi forms and the definition of the cycle integral $r$ see \cite{GKZ})
\begin{equation}\label{proof:shint}
\mathcal{S}^*_{\Delta, r}(\xi_{-2k}(f))= \left(\frac{i}{2N}\right)^{k} \sum_{\substack{n,r_0\in \Z\\r_0^2<4nN}}r_{k+1,N,\Delta(r_0^2-4nN),rr_0,\Delta}(\xi_{-2k}(f))q^n\zeta^{r_0}.
\end{equation}
Now by the Theorem and Corollary in Section II.4 in \cite{GKZ} we have
\begin{align*}
&\abs{ r_{k+1,N,\Delta(r_0^2-4nN),rr_0,\Delta}(\xi_{-2k}(f))}^2=\abs{\Delta}^{k+1/2}\abs{r_0^2-4nN}^{k+1/2} (k!)^2 N^{-k}
\\
&\quad\quad\quad\times 2^{-3k-2}\pi^{-2(k+1)} L(\xi_{-2k}(f),\Delta,k+1)\cdot L(\xi_{-2k}(f),r_0^2-4nN,k+1).
\end{align*}
Since $r_0$ and $n$ vary in \eqref{proof:shint} the Shintani lift $\mathcal{S}^*_{\Delta, r}(\xi_{-2k}(f))$ vanishes if and only if $L(\xi_{-2k}(f),\Delta,k+1)$ vanishes.
\end{proof}


Now we turn to the computation of the Fourier coefficients of positive index of the holomorphic part of the theta lift. 

Let $h\in L'/L$ and $m\in\Q_{>0}$ with $m \equiv \sgn(\Delta)Q(h)\ (\Z)$.
We define a twisted Heegner divisor on $M$ by
\[
Z_{\Delta,r}(m,h)= \sum\limits_{\lambda \in \G_0(N) \backslash L_{rh,m\abs{\Delta}}}\frac{\chi_{\Delta}(\lambda)}{\left|\overline\G_{\lambda}\right|} Z(\lambda).
\]
Here $\overline{\G}_\lambda$ denotes the stabilizer of $\lambda$ in $\overline{\G_0(N)}$.

Let $f$ be a harmonic weak Maa\ss{}  form of weight $-2k$ in $H^{+}_0(N)$. We put $\partial f:=R_{-2k}^kf(z)$.
Then the twisted modular trace function is defined as follows
\begin{equation}\label{def:trace11}
\mt_{\Delta,r}(f;m,h) = \sum\limits_{z\in Z_{\Delta,r}(m,h)}\partial f(z)=\sum\limits_{\lambda\in \G_0(N)\setminus L_{rh,\abs{\Delta}m}} \frac{\chi_{\Delta}(\lambda)}{\left|\overline\G_{\lambda}\right|}\partial f(D_{\lambda}).
\end{equation}

\begin{theorem}\label{thm:main}
We let $f\in H^{+}_{-2k}(N)$, $h\in L'/L$, and let $m\in \Q_{>0}$ with $m\equiv \sgn(\Delta) Q(h)\,(\Z)$. We obtain the following results:
\begin{enumerate}
 \item \label{thm:fceventwist}The $(m,h)$-th Fourier coefficient of the holomorphic part of $\LE$ equals
\[
 (-4\pi m)^{k/2}\, \mt_{\Delta,r}(f;m,h).
\]
 \item\label{thm:fcoddtwist} The $(m,h)$-th Fourier coefficient of the holomorphic part of $\LO$ equals
\begin{align*}
 \left(\frac{1}{4\pi m}\right)^{(k+1)/2} \prod_{j=0}^{(k-1)/2}\left(\frac{k}{2}+j\right)\left(j-\frac{k+1}{2}\right) \,\mt_{\Delta,r}(f;m,h) 
.
\end{align*}
\end{enumerate}
\end{theorem}

\begin{remark}
 For the case $k=0$ see  Theorem 4.5 in \cite{BrFu06} for the results on the untwisted trace and Theorem 5.5 in \cite{AE} for the twisted trace.
\end{remark}

\begin{proof}
To ease notation we prove the results when $\Delta=1$. Using the arguments of the proof of Theorem 5.5 in \cite{AE} it is straightforward to deduce the general result.

We first consider the Fourier expansion of $\int_M \partial f(z) \Theta(\tau,z,\phikm)$, namely
\begin{align}\label{proof:fc12}
\sum_{h\in L'/L}\sum_{m\in\Q}\left(\sum_{\lambda\in L_{m,h}}\int_M \partial f(z) \phikm^0(\sqrt{v}\lambda,z) \right)e^{2\pi im\tau}.
\end{align}
We denote the $(m,h)$-th coefficient of the holomorphic part of \eqref{proof:fc12} by $C(m,h)$.

Using the usual unfolding argument and then an argument of Katok and Sarnak \cite{KS} we find similarly as in the proof of Theorem 3.6 in \cite{BrOno2} 
\begin{equation}
 C(m,h)=\sum_{\lambda\in L_{m,h}} \frac{1}{|\bar{\G}_\lambda|}\partial f(\mathbb{D}_\lambda) Y_{\lambda}(\sqrt{mv}),
\end{equation}
where $\lambda=-k(k+1)$ and
\begin{equation}\label{proof:fc13}
 Y_{\lambda}(t)=4\pi \int_1^\infty \phikm^0(t\alpha(a)^{-1} \textbf{.} \lambda(i),i)\omega_\lambda(\alpha(a))\frac{a^2-a^{-2}}{2}\frac{da}{a}.
\end{equation}
Here $\omega(\alpha(a))$ is the standard spherical function of eigenvalue $\lambda$ and $\alpha(a)=\left(\begin{smallmatrix}a&0\\0&a^{-1}\end{smallmatrix}\right)$. Note that $\omega(\alpha(a))=\omega\left(\frac{a^2+a^{-2}}{2}\right)$.
Substituting $a=e^{r/2}$ we obtain that \eqref{proof:fc13} equals
\begin{equation*}
 2\pi \int_0^\infty \left(4t^2 \cosh(r)^2-\frac{1}{2\pi}\right) \omega_\alpha(\cosh (r)) \sinh(r) e^{-4\pi t^2\sinh(r)^2}dr.
\end{equation*}
In this case the standard spherical function is given by the Legendre polynomial $P_k(x)=\frac{1}{2^k k!} \frac{d^k}{dx^k}(x^2-1)^k$ \cite[Chapter 1]{Iwaniec}. By substituting $x=\sinh( r)^2$ we obtain
\begin{equation}\label{proof:fc15}
 4\pi t^2 \int_0^\infty \sqrt{1+x} P_k(\sqrt{1+x}) e^{-4\pi t^2 x} dx -\frac{1}{2} \int_0^\infty \frac{1}{\sqrt{1+x}} P_k(\sqrt{1+x}) e^{-4\pi t^2 x} dx.
\end{equation}

To evaluate the first integral in \eqref{proof:fc15} we use the following recursion formula for the Legendre polynomial (see for example equation (8.5.3) in \cite{Pocket})
\[
 \sqrt{1+x}P_k(\sqrt{1+x})=\frac{1}{2k+1}\left((k+1)P_{k+1}(\sqrt{1+x})+kP_{k-1}(\sqrt{1+x})\right).
\]
Thus, we are left with
\begin{equation}\label{proof:fc14}
4 \pi t^2 \int_0^\infty \left(\frac{k+1}{2k+1}  P_{k+1}(\sqrt{1+x}) + \frac{k}{2k+1}P_{k-1}(\sqrt{1+x})\right)e^{-4\pi t^2 x} dx,
\end{equation}
which is a Laplace transform computed in \cite{tables} (see equation (7) on page 180). It equals
\begin{equation}\label{proof:fc16}
(4\pi t^2)^{-1/4} e^{2\pi t^2} \left(\frac{k+1}{2k+1} W_{1/4,k/2+3/4}(4\pi t^2)+ \frac{k}{2k+1} W_{1/4,k/2-1/4}(4\pi t^2)\right).
\end{equation}

The second integral in \eqref{proof:fc15} can be evaluated in the same way (see equation (8) on page 180 of \cite{tables}) and equals
\begin{equation}\label{proof:fc17}
-\frac12 (4\pi t^2)^{-3/4} e^{2\pi t^2} W_{-1/4,k/2+1/4}(4\pi t^2).
\end{equation}

Using (13.1.33), (13.4.17), and (13.4.20) in \cite{Pocket} it is not hard to show that the sum of the expressions in \eqref{proof:fc16} and \eqref{proof:fc17} is equal to
\[
 e^{2\pi t^2}\mathcal{W}_{k/2+3/4,3/2}(4\pi t^2).
\]

Thus, $C(m,h)$ is given by
\[
 C(m,h)= \sum_{\lambda\in L_{m,h}} \frac{1}{|\bar{\G}_\lambda|}\partial f(\mathbb{D}_\lambda)  e^{2\pi mv}\mathcal{W}_{k/2+3/4,3/2}(4\pi m v).
\]

We now have to apply the iterated raising respectively lowering operator to the  Fourier expansion in \eqref{proof:fc12}, which boils down to evaluating it on
\[
 q^m  e^{2\pi mv}\mathcal{W}_{k/2+3/4,3/2}(4\pi m v)=\mathcal{W}_{k/2+3/4,3/2}(4\pi m v) e(mx).
\]

By Proposition \ref{prop:diffwhittw}  we obtain
\begin{align*}
R_{3/2,\tau}^{k/2}&\left(\mathcal{W}_{k/2+3/4,3/2}(4\pi m v) e(mx)\right)
\\
&=(-4\pi m)^{k/2} \mathcal{W}_{k/2+3/4,3/2+k}(4\pi m v) e(mx)
=(-4\pi m)^{k/2} q^m,
\end{align*}
since $W_{\nu,\mu}(y)= y^{k/2}e^{-y/2}$ for $y>0$, and $\nu=k/2$, $\mu=k/2/-1/2$ \cite[Chapter 13]{Pocket}. 

For the lowering operator a repeated application of Proposition \ref{prop:diffwhittw} yields
\begin{align*}
&L_{3/2,\tau}^{(k+1)/2} \left(\mathcal{W}_{k/2+3/4,3/2}(4\pi m v) e(mx)\right)
\\
& =\left(\frac{1}{4\pi m}\right)^{(k+1)/2}  \prod_{j=0}^{(k-1)/2}\left(\frac{k}{2}+j\right)\left(j-\frac{k+1}{2}\right)\mathcal{W}_{k/2+3/4,1/2-k}(4\pi m v).
\end{align*}
Again the Whittaker function simplifies, namely
$
  \mathcal{W}_{k/2+3/4,1/2-k}(4\pi m v)=e^{-2\pi mv},
$
which implies the desired statement.

\end{proof}

\section{Orthogonality to cusp forms}\label{sec:orthcusp}

In this section we show that $\LE\in M^{\text{!}}_{3/2,\widetilde{\rho}}$ is orthogonal to cusp forms with respect to the regularized Petersson inner product. Recall that for $g\in S_{3/2+k,\widetilde{\rho}}$ we have
\[
 \left(\LE, g(\tau)\right)_{3/2+k,\widetilde{\rho}}^{\text{reg}} = \lim_{t\rightarrow\infty} \int_{\F_t}\langle \LE,g(\tau) \rangle v^{3/2+k} d\mu(\tau),
\]
where $\F_t$ denotes the truncated fundamental domain $\F_t=\left\{\tau\in\h; \Im(\tau)\leq t\right\}$.

\begin{theorem}\label{thm:orthcusp}
 For $\LE\in M^{\mathrm{!}}_{3/2+k,\widetilde{\rho}}$, where $k\geq 0$, and $g\in S_{3/2+k,\widetilde{\rho}}$ we have
\[
 \left(\LE, g(\tau)\right)_{3/2+k,\widetilde{\rho}}^{\text{reg}}=0.
\]
\end{theorem}

Since $\xi_k(F)$ is a cusp form for a harmonic Maa\ss{}  form $F$, Theorem \ref{thm:orthcusp} together with Proposition \ref{prop:extpairing} directly implies
\begin{corollary}\label{thm:duality}
We let $F$ be a harmonic weak Maa\ss{}  form of weight $\kappa$ transforming with representation $\overline{\widetilde{\rho}}$. Here $\kappa= 3/2+k$, if $k$ is odd, and $\kappa=1/2-k$, if $k$ is even. We denote the $(m,h)$-th Fourier coefficient of the holomorphic part by $c_F(m,h)$. Moreover, let $f\in M^{\text{!}}_{-2k}(N)$, such that $\Lambda_{\Delta,r}(\tau,f)$ is weakly holomorphic and transforms with representation $\widetilde{\rho}$. We denote the Fourier coefficients of the principal part by $a_\Lambda(m,h)$. Then we have
\begin{align*}
 \sum_{h\in L'/L}& \sum_{\substack{m\geq 0\\m\equiv \sgn(\Delta)Q(h)\, (\Z)}}  c_F^+(-m,h) \mt_{\Delta,r}(f;m,h)
\\
&=-\sum_{h\in L'/L}\sum_{\substack{m\geq 0\\ -N\abs{\Delta}m^2\equiv \sgn(\Delta)Q(h)\, (\Z)}}c_F^+(N\abs{\Delta}m^2,h)a_{\Lambda}^+(-N\abs{\Delta}n^2,h).
\end{align*}
\end{corollary}

\begin{proof}[Proof of Theorem \ref{thm:orthcusp}]
  To ease notation we prove the theorem in the untwisted case. Since the twisted lift is essentially a linear combination of untwisted ones the arguments carry over directly (see the proof of \cite[Theorem 5.5]{AE}).

Using the dominated convergence theorem it is tedious but straightforward to show that interchanging the integration with respect to $z$ and $\tau$ is allowed. That is
\begin{align*}
& \lim_{t\rightarrow\infty} \int_{\F_t}\langle \LE,g(\tau) \rangle v^{3/2+k} d\mu(\tau)
\\
&= \lim_{t\rightarrow\infty}  \int_{\F_t}\langle R_{3/2,\tau}^{k/2} \int_M (R^{k}_{-2k,z}f)(z)\thetaL{\phikm}, g(\tau)\rangle
\\
&= \int_M (R^{k}_{-2k,z}f)(z) \lim_{t\rightarrow\infty} \int_{\F_t} \langle R_{3/2,\tau}^{k/2}\thetaL{\phikm}, g(\tau)\rangle.
\end{align*}

We now consider the cases $k=0$ and $k>0$ separately. 
We first show that for $k>0$
\[
  \lim_{t\rightarrow\infty} \int_{\F_t} \langle R^{k/2}_{3/2,\tau} \Theta(\tau,z,\phikm),g(\tau)\rangle v^{3/2+k} d\mu(\tau)=0.
\]
Following the proof of \cite[Theorem 4.1]{BOR} we let 
\[
H:=v^{k-1/2}\overline{R_{3/2,\tau}^{k/2-1} \Theta(\tau,z,\phikm)}
\]
and
\[
h:=R_{3/2,\tau}^{k/2} \Theta(\tau,z,\phikm)=v^{-k-3/2}\overline{L_{1/2-k,\tau}H}.
\] 
Note that $R^{k/2-1}_{3/2,\tau} \Theta(\tau,z,\phikm)$ is only defined for $k>0$. 
We obtain
\begin{align*}
 &\int_{\F_t} \langle R_{3/2,\tau}^{k/2}  \Theta(\tau,z,\phikm),g(\tau) \rangle v^{3/2+k} d\mu(\tau)
\\
 &\quad\quad\quad\quad= \int_{\F_t} \langle v^{-k-3/2}\overline{L_{1/2-k,\tau}H} ,g(\tau) \rangle v^{3/2+k} d\mu(\tau).
\end{align*}
We have that
\begin{align}
 \notag  \langle v^{-k-3/2}\overline{L_{1/2-k,\tau}H} ,g(\tau) \rangle v^{3/2+k} d\mu(\tau)
 &\notag=  \langle \overline{2iv^2 \frac{\partial}{\partial\overline{\tau}}H} ,g(\tau) \rangle  \frac{dudv}{v^2}
 \\
 &\label{proof:der}= -\langle \overline{\frac{\partial}{\partial\overline{\tau}}H(\tau)},g(\tau) \rangle d\tau d\overline{\tau}.
\end{align}
By the holomorphicity of $g$ we obtain that \eqref{proof:der} equals
\[
 -\langle \overline{\frac{\partial}{\partial\overline{\tau}}H(\tau)},g(\tau) \rangle d\tau d\overline{\tau}=- \partial\left(\langle \overline{H(\tau)},g(\tau)\rangle d\overline{\tau}\right)= -d\left(\langle \overline{H(\tau)},g(\tau)\rangle d\overline{\tau}\right).
\]

We now apply Stoke's Theorem.
Since the integrand is $\mathrm{SL}_2(\Z)$-invariant the equivalent pieces of the boundary of the fundamental domain cancel and we obtain
\begin{align}
 & \int_{\F_t} \langle R_{3/2,\tau}^{k/2}  \Theta(\tau,z,\phikm),g(\tau) \rangle v^{3/2+k} d\mu(\tau)\notag
 \\
 &=-\int_{\partial\F_t} \langle v^{k-1/2}R_{3/2,\tau}^{k/2-1} \Theta(\tau,z,\phikm),g(\tau) \rangle d\tau\notag
 \\
 &= \sum_{h\in L'/L} \int_{-1/2}^{1/2} t^{k-1/2} R_{3/2,\tau}^{k/2-1}\theta_h(u+it,z,\phikm)\label{eq:dualityk} \overline{g_h(u+it)}du. 
\end{align}
Plugging in the Fourier expansions of the two series and carrying out the integration over $u$ we see that \eqref{eq:dualityk} equals
\[
\sum_{h\in L'/L} t^{k-1/2} \sum_{n=1}^\infty b(n,h) a(n,h) e^{-4\pi n t},
\] 
where $b(n,h)$ and $a(n,h)$ denote the Fourier coefficients of $g_h$ and $R_{3/2,\tau}^{k/2-1}\theta_h$. Here, the main contribution comes from the exponential terms, implying that the limits tends to $0$ as $t\rightarrow \infty$. 

For $k=0$ we use an argument for harmonic forms on Riemann surfaces to show that
\[
  \lim_{t\rightarrow\infty} \int_{\F_t}   \thetah{\phikm} \overline{g_h(\tau)} v^{3/2} d\mu(\tau)=0,
\]
where we consider the components separately now.

We first show that $\Delta_{0,z}$ annihilates this expression.
Since the the partial derivatives $\frac{\partial^2}{\partial x^2}$ and $\frac{\partial^2}{\partial y^2}$ of 
$
 \int_{\F_t} \thetah{\phikm}\overline{g_h(\tau)} v^{3/2} d\mu(\tau)
$
converge locally uniformly in $z$ as $t\rightarrow \infty$, we can interchange differentiating and taking the limit.

Recall that we have $
 \Delta_{3/2,\tau} \Theta(\tau,z,\phikm)= \frac14 \Delta_{0,z}\Theta(\tau,z,\phikm)
$ by \eqref{eq:deltatheta}, which implies
\begin{align*}
& \int_{\F_t} \Delta_{0,z}\thetah{\phikm}\overline{g_h(\tau)} v^{3/2} d\mu(\tau)
\\
&\quad\quad\quad\quad=
4 \int_{\F_t} \Delta_{3/2,\tau}\thetah{\phikm}\overline{g_h(\tau)} v^{3/2} d\mu(\tau).
\end{align*}
By Lemma 4.3 of \cite{BrHabil} we find
\begin{align}
&\int_{\F_t} \Delta_{3/2,\tau}\thetah{\phikm}\overline{g_h(\tau)} v^{3/2} d\mu(\tau) \notag
\\
&\quad\quad = \int_{\F_t} \thetah{\phikm}\Delta_{3/2,\tau}\overline{g_h(\tau)} v^{3/2} d\mu(\tau)\label{eq:int1}
\\
&\quad\quad\quad +\int_{-1/2}^{1/2} \left[\thetah{\phikm} \overline{L_{3/2,\tau}g_h(\tau)}v^{3/2}\right]_{v=t}du \label{eq:int2}
\\
&\quad\quad\quad - \int_{-1/2}^{1/2} \left[ L_{3/2,\tau}\thetah{\phikm}\overline{g_h(\tau)}v^{3/2}\right]_{v=t}du.\label{eq:int3}
\end{align}
The holomorphicity of $g$ implies that the integrals in \eqref{eq:int1} and \eqref{eq:int2} vanish. When plugging in the Fourier expansions of  $g_h(u+it)$ and $L_{3/2,\tau}\theta_h(u+it,z,\phikm)$ and integrating over $u$ we see that the resulting expression is exponentially decaying as $t\rightarrow \infty$, which then implies 
\[\Delta_{0,z} \lim_{t\rightarrow\infty} \int_{\F_t}   \thetah{\phikm} \overline{g_h(\tau)} v^{3/2} d\mu(\tau)=0.\]

Writing
$
\lim_{t\rightarrow\infty}\int_{\F_t} \thetah{\phikm}\overline{g_h(\tau)} v^{3/2} d\mu(\tau)=h(z)d\mu(z)
 $
for a smooth function on $M$, we have
$
\Delta_{0,z} h(z)=0
$.
By the square-exponential decay of the Kudla-Millson theta function \cite[Proposition 4.1]{Funke} $\Delta_{0,z} h(z)=0$ implies that $h(z)$ is constant \cite[Corollary 4.22]{BrHabil}. So it remains to show that this constant is zero.  
This follows from the following:
For $z\in\h$ and $\sigma_{\ell}$ as in Section \ref{sec:prelim} we have
\begin{equation}\label{eq:harm}
  \lim_{y\rightarrow i\infty}h(\sigma_{\ell}z)=0.
\end{equation}
For simplicity, we only consider the cusp $\ell=\infty$. A careful analysis yields that we can interchange the limit processes with respect to $t$ and $y$. The square exponential decay of $\thetah{\phikm}$ implies that $\lim_{y\rightarrow i\infty}\thetah{\phikm}=0$. Therefore, $\lim_{y\rightarrow i\infty}h(y)$ vanishes.

\end{proof}


\section{Singular Moduli}\label{sec:singmod}

Using Bruinier's and Ono's results \cite[Section 4]{BrOno2} we show that the singular moduli $3^{k} D^k (\frac{1}{\pi})^k R_{-2k}^k f(\alpha_Q)$ are algebraic integers. We first fix notation. We let $-D<0$ be a discriminant and $r\in\Z$ such that $r^2\equiv -D\pmod{4N}$. By $\mathcal{Q}_{D,r,N}$ we denote the set of positive definite integral binary quadratic forms $[a,b,c]$ of discriminant $-D$ with $N|a$ and $b\equiv r\pmod{2N}$. Then we let $\alpha_Q=\frac{-b+\sqrt{-D}}{2a}$ be the Heegner point corresponding to $Q\in\mathcal{Q}_{D,r,N}$. We write $\mathcal{O}_D$ for the order of discriminant $-D$ in $\Q(\sqrt{-D})$.

\begin{theorem}\label{thm:singmoduli}
Let $N$ be a square-free integer and let $D>0$ be coprime to $2N$ and $r\in\Z$ with $r^2\equiv -D \pmod{4N}$. Assume that $f\in M^{\text{!}}_{-2k}(N)$ has integral coefficients at all cusps. If $Q\in \mathcal{Q}_{D,r,N}$ is primitive, then $3^{k} D^k (\frac{1}{\pi})^k R_{-2k}^k f(\alpha_Q)$ is an algebraic integer in the ring class field for the order $\mathcal{O}_D \subset \Q(-\sqrt{D})$.
\end{theorem}
\begin{proof}
 The theorem follows from combining work of Miller and Pixton and of Bruinier and Ono. In \cite[Proposition 3.1]{MilPix} Miller and Pixton show the algebraicity of the singular moduli $R_{-2k}^kf$ and the obvious generalization of \cite[Lemma 4.7]{BrOno2} yields the bound on the denominators.
\end{proof}

\begin{remark}
By classical results on singular moduli (the results that are relevant for us are summarized in \cite[Theorem 4.1]{BrOno2}) it follows that the multiset of values  $ R_{-2k}^k f(\alpha_Q)$ is a union of Galois orbits.
\end{remark}
\begin{remark}
Using work of Larson and Rolen \cite{LarRol} one might be able to drop the $3^k$.
\end{remark}

\section{Examples}
Here we prove the formulas for the coefficients of $\eta(\tau)^{-25}$ and Ramanujan's mock theta function $f(q)$ that were presented in the introduction.

\subsection{Powers of the Dedekind $\eta$-function}\label{sec:eta}

We let
\[
 G_{25}(\tau):=\sum_{r\in\Z/ 12\Z} \chi_{12}(r) \eta(\tau)^{-25}\mathfrak{e}_r.
\]
Moreover, we define 
\begin{align*}
 F&=-F_5(\cdot, 14, -26) + F_5(\cdot, 14, -26) |W_2^6 
\\
&\quad\quad+ F_5(\cdot, 14, -26) |W_3^6 - F_5(\cdot, 14, -26) |W_6^6
\end{align*}
and 
\begin{align*}
 \tilde{F} &=(25+5^{13}) \left(F_1(\cdot, 14, -26) - F_1(\cdot, 14, -26)|W_2^6 \right.
\\
&\quad\quad\left. - F_1(\cdot,14, -26)|W_3^6 + F_1(\cdot,14, -26)|W_6^6\right).
\end{align*}

Here we write $W_Q^N$ for the corresponding Atkin-Lehner involution (see for example \cite[Chapter IX.7]{Knapp} for a definition).

\begin{corollary}
 For $n>0$ the coefficient of index $\left(\frac{24n-1}{24},1\right)$ of $G_{25}$ is given by
\begin{align*}
&-\frac{185725}{4429185024\pi^{13}}\left(\frac{1}{24n-1}\right)^7
\\
&\quad\quad\quad\quad\quad\quad\quad\times\left(\mt\left(F;\frac{24n-1}{24},1\right)+ \mt\left(\tilde{F};\frac{24n-1}{24},1\right)\right).
\end{align*}
\end{corollary}

\begin{proof}
Using the transformation properties of the Dedekind $\eta$-function one easily sees that $G_{25}\in M^{\text{!}}_{-25/2,\rho}$.
The principal part of $G_{25}(\tau)$ is equal to $(q^{-25/24}+25q^{-1/24})(\mathfrak{e}_1-\mathfrak{e}_5-\mathfrak{e}_7+\mathfrak{e}_{11})$.

For the lift $\frac{1}{C^{\text{o}}}\Lambda^{\text{o}}(\tau,F)$ of the Poincar\'{e} series $F$ we obtain (where $C^{\text{o}}$ is as in Theorem \ref{thm:liftpoincare})
\begin{align*}
& -  \sum_{n|5}n^{13} \mathcal{F}_{\frac{25}{24n^2},\frac{5}{n}}\left(\tau,\frac{29}{4},-\frac{25}{2} \right)+  \sum_{n|5}n^{13} \mathcal{F}_{\frac{25}{24n^2},\frac{5}{n}}\left(\tau,\frac{29}{4},-\frac{25}{2} \right)|W_2^6
\\
&+ \sum_{n|5}n^{13} \mathcal{F}_{\frac{25}{24n^2},\frac{5}{n}}\left(\tau,\frac{29}{4},-\frac{25}{2} \right)|W_3^6 - \sum_{n|5}n^{13} \mathcal{F}_{\frac{25}{24n^2},\frac{5}{n}}\left(\tau,\frac{29}{4},-\frac{25}{2} \right)|W_6^6.
\end{align*}  
This has principal part $2(q^{-25/24}-5^{13}q^{-1/24})(\mathfrak{e}_1-\mathfrak{e}_5-\mathfrak{e}_7+\mathfrak{e}_{11})$.

The lift $\frac{1}{C^{\text{o}}} \Lambda^{\text{o}}(\tau,\tilde{F})$ of $\tilde{F}$ is given by
\begin{align*}
& (25+5^{13})\left( \mathcal{F}_{\frac{1}{24},1}\left(\tau,\frac{29}{4},-\frac{25}{2} \right)- \mathcal{F}_{\frac{1}{24},1}\left(\tau,\frac{29}{4},-\frac{25}{2} \right)|W_2^6\right.
\\
&\quad\quad\quad\quad \left. - \mathcal{F}_{\frac{1}{24},1}\left(\tau,\frac{29}{4},-\frac{25}{2} \right)|W_3^6 + \mathcal{F}_{\frac{1}{24},1}\left(\tau,\frac{29}{4},-\frac{25}{2} \right)|W_6^6\right).
\end{align*}
This has principal part $2  (25+5^{13})q^{-1/24}(\mathfrak{e}_1-\mathfrak{e}_5-\mathfrak{e}_7+\mathfrak{e}_{11})$. Then the sum 
\[
\frac{1}{2C^{\text{o}}}\left(\Lambda^{\text{o}}(\tau,F)+\Lambda^{\text{o}}(\tau,\tilde{F})\right)
\] has principal part $(q^{-25/24}+25q^{-1/24})(\mathfrak{e}_1-\mathfrak{e}_5-\mathfrak{e}_7+\mathfrak{e}_{11})$.
Thus,
\[
 G_{25}(\tau)=\frac{1}{2 C^{\text{o}}} (\Lambda^{\text{o}}(\tau,F)+\Lambda^{\text{o}}(\tau,\tilde{F})),
\]
which implies the formula in the corollary.

\end{proof}

\begin{remark}
 More generally, one can deduce formulas for the coefficients of $\eta(\tau)^{-i}$, where $i\equiv 1 \pmod{24}$.
Here we let
\[
 G_{i}(\tau):=\sum_{r\in\Z/ 12\Z} \chi_{12}(r) \eta(\tau)^{-i}\mathfrak{e}_r.
\]
Then, similarly as above, one has to construct a linear combination of twisted lifts of Poincar\'{e} series whose lift has the same principal part as $G_i(\tau)$.
\end{remark}

\subsection{A formula for the coefficients of $f(q)$}\label{sec:fq}

We consider Ramanujan's mock theta function
\[
f(q):= 1+ \sum\limits_{n=1}^\infty \frac{q^{n^2}}{\left(1+q\right)^2 \left(1+q^2\right)^2\cdots\left(1+q^n\right)^2 }=1+ \sum\limits_{n=1}^\infty a_f(n)q^n.
\]
We let $\Delta<0$ be a fundamental discriminant with $\Delta \equiv 1 \pmod{24}$. We define
\[
F(z)= -\frac{1}{40} \frac{E_4(z)+ 4E_4(2z)-9E_4(3z)- 36E_4(6z)}{\eta(z)^2\eta(2z)^2\eta(3z)^2\eta(6z)^2}= q^{-1}-4+83q+\cdots,
\]
which is a weakly holomorphic modular form of weight $0$ for $\G_0(6)$.

\begin{corollary}
We have
\begin{align*}
 a_f\left(\frac{\abs{\Delta}+1}{24}\right) = -\frac{1}{8i\sqrt{\abs{\Delta}}}&(\mt_{\Delta,1}(F;1,1)-\mt_{\Delta,1}(F;1,5)
\\
&\quad
+\mt_{\Delta,1}(F;1,7)-\mt_{\Delta,1}(F;1,11)).
\end{align*}
\end{corollary}
\begin{remark}
 These formulas were checked numerically by Stephan Ehlen.
\end{remark}

\begin{proof}
Here we employ the duality results between weight $1/2$ and $3/2$. 
Note that $q^{-1/24}f(q)$ can be realized as the component of the holomorphic part of a vector valued harmonic Maa\ss{}  form $\tilde{H}$ of weight $1/2$ with representation $\bar{\rho}$ \cite[Lemma 81]{BrOno}. More precisely, 
\[
H=(0,h_0,h_2-h_1,0,-h_1-h_2,-h_0,0,h_0,h_1+h_2,0,h_1-h_2,-h_0)^{\text{T}},
\]
where the holomorphic part of $h_0$ is $q^{-1/24}f(q)$ and the holomorphic parts of $h_1$ and $h_2$ are given by Ramanujan's mock theta function $\omega(q)$. The non-holomorphic parts are given by certain unary theta series \cite[Section 8.2]{BrOno}.

The principal part of $H$ is given by
$
q^{-1/24}(\e_1-\e_5+\e_7-\e_{11})$.
In terms of Poincar\'{e} series we have
\[
F(z)\doteq F_1(z,1,0)+ F_1(z,1,0)|W_2^6- F_1(z,1,0)|W_3^6- F_1(z,1,0)|W_6^6,
\]
where $\doteq$ means up to addition of a constant. 

By Theorem 5.5 and Theorem 5.6 of \cite{AE} we see that $\Lambda_{\Delta,1}^{\text{e}}(\tau,F)$ is a weakly holomorphic modular form. Note that the non-holomorphic part vanishes since $6$ is square-free and for $\Delta<0$ we have $\chi_\Delta(-\lambda)=-\chi_\Delta(\lambda)$.
To determine the principal part of $\Lambda_{\Delta,1}^{\text{e}}(\tau,F)$ we compute the lift of the Poincar\'{e} series. Note that the lift of a constant vanishes.

By Theorem \ref{thm:liftpoincare} the function $F(z)$ lifts to a vector valued Poincar\'{e} series having principal part 
\[
 2i\abs{\Delta}^{1/2} q^{-\abs{\Delta}/24} (\e_1-\e_5+\e_7-\e_{11}).
\]
Therefore, by Proposition \ref{prop:extpairing} we obtain that $\left\{\Lambda_{\Delta,1}^{\text{e}}(\tau,F),H\right\}=0$, which implies
\begin{align*}
&c_H^+(-1,1)\mt_{\Delta,1}(F;1,1)+c_H^+(-1,5)\mt_{\Delta,1}(F;1,5)
\\
&\quad\quad+c_H^+(-1,7)\mt_{\Delta,1}(F;1,7)+c_H^+(-1,11)\mt_{\Delta,1}(F;1,11)
\\
&=-2i\sqrt{\abs{\Delta}} \left(c_H^+\left(\frac{\abs{\Delta}}{24},1\right)\cdot 1+c_H^+\left(\frac{\abs{\Delta}}{24},5\right)\cdot (-1)\right.
\\
&\quad\quad\left.+c_H^+\left(\frac{\abs{\Delta}}{24},7\right)\cdot 1+c_H^+\left(\frac{\abs{\Delta}}{24},11\right)\cdot (-1)\right).
\end{align*}

Since we can identify the coefficients in the different components of $H$ we obtain the formula in the corollary.

\end{proof}

\begin{remark}
 For $\Delta\equiv r^2\pmod{24}$, where $r\equiv 5,7,11\pmod{12}$, we consider 
\[
F(z)= -\frac{1}{40} \frac{E_4(z)\pm 4E_4(2z)\pm9E_4(3z)\pm 36E_4(6z)}{\eta(z)^2\eta(2z)^2\eta(3z)^2\eta(6z)^2}= q^{-1}-4+83q+\cdots
\]
and have  to arrange the $\pm$'s in such a way that we obtain (up to a constant) $q^{-\abs{\Delta}/24}(\e_1-\e_5+\e_7-\e_{11})$ as the principal part of the lift.
\end{remark}


\bibliography{bib_complete.bib}
\bibliographystyle{amsalpha}

\end{document}